\documentclass[a4, 12pt]{amsart}
\usepackage{amssymb,amstext,amsmath,amscd,amsthm,amsfonts,enumerate,latexsym}

\newtheorem{thm}{Theorem}[section]
\newtheorem{prop}[thm]{Proposition}
\newtheorem{lem}[thm]{Lemma}
\newtheorem{cor}[thm]{Corollary}
\newtheorem{claim}[thm]{Claim}

\newtheorem{ex}[thm]{Example}

\numberwithin{equation}{thm}

\def\Ext{\mathrm{Ext}}

\newcommand{\Soc}{\mathop{\mathrm{Soc}}\nolimits}

\def\mod{\mathrm{mod}}

\newcommand{\Ker}{\mathop{\mathrm{Ker}}\nolimits}
\newcommand{\im}{\mathop{\mathrm{Im}}\nolimits}
\newcommand{\Coker}{\mathop{\mathrm{Coker}}\nolimits}
\newcommand{\depth}{\mathop{\mathrm{depth}}\nolimits}
\newcommand{\Ass}{\mathop{\mathrm{Ass}}\nolimits}

\def\e{\mathrm e}
\def\m{\mathfrak m}
\def\n{\mathfrak n}
\def\k{\kappa}

\def\p{\mathfrak p}
\def\q{\mathfrak q}

\def\w{\omega}
\def\P{\mathcal P}

\def\Z{\Bbb Z}

\def\K{\mathrm{K}}
\def\H{\mathrm{H}}

\def\th{\mbox{\tiny th}}

\newcommand{\calC}{\mathcal{C}}

\newcommand{\calK}{\mathcal{K}}

\newcommand{\fka}{\mathfrak{a}}
\newcommand{\fkb}{\mathfrak{b}}
\newcommand{\fkc}{\mathfrak{c}}

\newcommand{\fkm}{\mathfrak{m}}

\newcommand{\fkq}{\mathfrak{q}}

\newcommand{\fkM}{\mathfrak{M}}

\def\Supp{\mathrm{Supp}}

\def\height{\mathrm{ht}}

\def\G{{\mathcal G}}
\def\R{{\mathcal R}}

\def\M{{\mathcal M}}

\title[Gorensteinness in Rees algebras of powers of parameter ideals]{Gorensteinness in Rees algebras of powers of parameter ideals}

\author[Shiro Goto]{Shiro Goto}
\address{Department of Mathematics, School of Science and Technology, Meiji University, 1-1-1 Higashi-mita, Tama-ku, Kawasaki 214-8571, Japan}
\email{shirogoto@gmail.com}

\author[Shin-ichiro Iai]{Shin-ichiro Iai}
\address{Mathematics laboratory, Sapporo College, Hokkaido University of Education, 1-3 Ainosato 5-3, Kita-ku, Sapporo 002-8502, Japan}
\email{iai.shinichiro@s.hokkyodai.ac.jp}

\thanks{The first author was partially supported by JSPS Grant-in-Aid for Scientific Research (C) 21K03211. }

\begin{document}

\maketitle

\begin{abstract}
This paper gives a necessary and sufficient condition for Gorensteinness in Rees algebras of the 
$d^{\th}$ 
power of parameter ideals in certain Noetherian local rings of dimension $d\ge 2$. The main result of this paper produces many Gorenstein Rees algebras over non-Cohen-Macaulay local rings. For example, the Rees algebra $\R(\q^d)=\oplus_{i\ge 0}\q^{di}$ is Gorenstein for every parameter ideal $\q$ that is a reduction of the maximal ideal in a $d$-dimensional Buchsbaum local ring of depth 1 and multiplicity 2.
\end{abstract}

\section{Introduction}
Throughout this paper, all rings are assumed to be commutative with nonzero identity. Let $A$ be a Noetherian local ring with maximal ideal $\m$ and $d=\dim A>0$. For each ideal $I$ of $A$, we set $\R(I)=A[It]~(\subseteq A[t])$, which is called the Rees algebra of $I$, where $t$ is an indeterminate over $A$. 

The question of when there exists an ideal $I$ such that $\R(I)$ becomes a Cohen-Macaulay ring has come to fruition as Kawasaki's arithmetic Cohen-Macaulayfication (\cite{K}). 
On the other hand, when the base ring $A$ is not necessarily Cohen-Macaulay, studies on the structure of rings containing an ideal $I$ such that $\R(I)$ becomes a Gorenstein ring seem to be not still making sufficient progress.
The purpose of this paper is to  provide a necessary and sufficient condition for Gorensteinness of Rees algebras of certain ideals in non-Cohen-Macaulay local rings. Let $a_1, a_2, \dots, a_d$ be a system of parameters of $A$ and put $\fkq=(a_1, a_2, \dots, a_d)$.  
We actually investigate the Rees algebra of the $d^{\th}$ power of $\q$, that is to say, $\R(\q^d)$. 
To state our results, let us fix the following notation.
For each $A$-module $M$, let $\H^i_\m(M)$ be the $i^{\th}$ local cohomology module of $M$ with respect to $\m$ and $\ell_A (M)$ the length of $M$ as an $A$-module. We put $$\mathfrak{c}=(0):_A \H^1_\m(A),$$ which is an $\mathfrak{m}$-primary ideal whenever $0<\ell_A (\H^1_\m(A))<\infty$. For each $\mathfrak{m}$-primary ideal $I$ of $A$, we denote the multiplicity of $I$ by $\e_I(A)$. 
For each nonzero finitely generated $A$-module $M$, set $\mathrm{r}_A(M)=\ell_A (\mathrm{Ext}^{\depth_AM}_{A}(A/\mathfrak{m},M))$ which is called the type of $M$. With this notation, the main result in the paper can be stated as follows.


\begin{thm}\label{main}
Assume that $d\ge 2$, $\H^i_\m(A)=(0)$ for $i\neq 1,d$, and $\ell_A (\H^1_\m(A))<\infty$. Then the following three conditions are equivalent.\vspace{1mm}
\begin{enumerate}
	\item[$(1)$] The Rees algebra $\R(\fkq^d)$ is Gorenstein.
\vspace{2mm}
	\item[$(2)$] $\H^1_\m(A)\neq (0)$, $\mathrm{r}_A(\H^1_\m(A))=1$, and $\fkc=\sum_{i=1}^d(a_1,\dots, a_{i-1}, a_{i+1},\dots a_d)A:_A a_i$.
\vspace{2mm}
	\item[$(3)$] $\depth A=1$, $\mathrm{r}_A(A)=1$, $\e_\fkc (A)=2\ell_A (A/\fkc)$, and $\q$ is a reduction of $\fkc$.\vspace{1mm}
\end{enumerate}
When this is the case, the $(S_2)$-ification $\widetilde{A}$ of $A$ is a Gorenstein ring
with the equalities $\fkc=\q\widetilde{A}=A:\widetilde{A}$ and $\ell_A (\widetilde{A}/\fkc) = 2\ell_A (A/\fkc)$.
\end{thm}

The $(S_2)$-ification $\widetilde{A}$ of $A$ is defined as the smallest module-finite birational extension of $A$ satisfying the condition $(S_2)$ of Serre. The colon $$A:\widetilde{A}=\{f\in\mathrm{Q}(A) \mid f\widetilde{A}\subseteq A\}$$ is considered inside the total ring $\mathrm{Q}(A)$ of fractions of $A$. The assumption $\H^i_\m(A)=(0)$ for $i\neq 1, d$ is equivalent to saying that the $(S_2)$-ification $\widetilde{A}$ of $A$ is a Cohen-Macaulay ring when $\ell_A (\H^i_\m(A))<\infty$ for $i\neq d$ (cf. \cite{G80}). Thanks to \cite[Lemma 3.1]{GNa}, the assumption $\ell_A (\H^1_\m(A))<\infty$ is naturally satisfied when the base ring $A$ is unmixed $($i.e., $\dim \widehat{A}=\dim \widehat{A}/P$ for all $P\in\Ass\widehat{A}$, where $\widehat{A}$ denotes the $\m$-adic completion$)$.

The origin of a study of $\R(\q^d)$ comes from an example of Hochster and Roberts in \cite{HR} as follows. Put $A=k[[x^2,y,x^3,xy]]$ that is a subring of the formal power series ring $k[[x,y]]$ over a field $k$ in two variables $x$ and $y$. We set $\fkq=(x^2, y)A$. 
Then they showed that the Rees algebra $\R(\q^2)$ is Gorenstein but the base ring $k[[x^2,y,x^3,xy]]$ is not Cohen-Macaulay. It is one of the most important examples of Rees algebras and has provided the impetus for large amount of research. It is not only an example of a non-Cohen-Macaulay ring that is a direct summand of a Gorenstein ring, but also an example of an arithmetic Gorensteinification. 

As far as we know, the first result of Gorensteinness of $\R(\fkq^d)$ was reported by Yasuhiro Shimoda around 30 years ago in his seminar talk at Meiji University as follows.

\begin{thm}[Shimoda]
Assume that $\dim A= 2$. Let $\q=(a, b)$ be a parameter ideal of $A$. Then the following two conditions are equivalent.\vspace{1mm}
\begin{enumerate}
	\item[$(1)$] The Rees algebra $\R(\fkq^2)$ is Gorenstein.
\vspace{2mm}
	\item[$(2)$] 
\begin{enumerate}
\item[$\mathrm{(i)}$] Both $a$ and $b$ are non-zerodivisors on $A$.
\item[$\mathrm{(ii)}$] The equality $(aA:_Ab)\cap (bA:_Aa)=aA\cap bA$ holds.
\item[$\mathrm{(iii)}$] The ring $A/abA+a(aA:_Ab)+b(bA:_Aa)$ is Gorenstein.
\end{enumerate}
\end{enumerate}
\end{thm}

The present paper is motivated his argument and gives a complete generalization of his theorem together with many other results.

We notice that the Gorensteinness of $\R(\q^d)$ implies that the ideal $\q$ is generated by a {\it standard} system of parameters (cf. \cite[p.261]{SV} for the definition) under the assumption in Theorem \ref{main} (see Lemma \ref{standard} in the present paper). Furthermore,
Herrmann, Hyry, and Korb proved that the base ring $A$ must be Gorenstein whenever $\depth A\ge 2$ and $\R(\q^n)$ is a Gorenstein ring for some positive integer $n$ and for some standard parameter ideal $\q$ (\cite[Corollary 4.7]{HHK}). When $A$ is a Gorenstein local ring with $\dim A\ge 2$, the characterization of Gorensteinness of $\R(\q^n)$ is already known (cf. \cite[Theorem (1.2)]{GS2}, \cite[Lemma 2.4]{Hoa}, and \cite[Theorem 4.1]{O}). Thus we focus on the case where $\depth A=1$ in the present paper. 

We will show that $\R(\q^n)$ is not a Gorenstein ring for any positive integer $n\neq d$ whenever $d\ge 2$, $\depth A=1$, and $\q$ is a standard parameter ideal (see Theorem \ref{n=d}). This is the reason why we consider the $d^{\th}$ power of $\q$.

For Theorem \ref{main}, we are interested in the Gorensteinness of $\R(\q^d)$ deeply influencing properties of the base ring $A$, namely, $\depth A=1$, $\mathrm{r}_A(A)=1$, and $\e_\fkc (A) = 2\ell_A (A/\fkc)$. We shall show that these equalities imply that the $(S_2)$-ification $\widetilde{A}$ of $A$ is a Gorenstein ring (see Theorem \ref{gor}). 

As a consequence of Theorem \ref{main}, we have the following explicit result when $A$ is a Buchsbaum local ring.

\begin{cor}\label{Bbm}
Assume that $d\ge 2$. Let $A$ be a Buchsbaum local ring of depth one. Then the Rees algebra $\R(\fkq^d)$ is Gorenstein if and only if $\e_\m (A)=2$ and $\q$ is a reduction of $\m$.
When this is the case, one has $\H^i_\m(A)=(0)$ if $i\neq 1,d$.
\end{cor}

Let us explain how to organize this paper.
The proofs of Theorem \ref{main} and Corollary \ref{Bbm} will be given in Section 5.
There are two notions which will play essential roles in investigating the Gorenstein property of $\R(\q^d)$ in this paper. 
One is the $(S_2)$-ification $\widetilde{A}$ obtained by an unconditioned strong $d$-sequence due to the first author and Yamagishi in \cite{GY}. It was essentially first investigated in \cite{GS} and was recognized as the module-finite Cohen-Macaulayfication under the suitable assumptions in \cite{G80}. Section 2 shall be devoted to studying it. The other is a filtration describing the canonical module of the Rees algebra that was originally introduced by Herzog, Simis, and Vasconcelos in \cite{HSV}. It has been fruitfully developed by several authors (see, e.g., \cite{Z,HHR,TVZ,U,GI,H}, in which \cite{Z} and \cite{TVZ} are typical references given by Zarzuela and by Trung, Vi\^et, and Zarzuela).
Let us refer to such a filtration as a canonical filtration (cf. \cite{GI}). In order to use it as a tool to prove Theorem \ref{main}, we will build structures in which two canonical filtrations with respect to $A$ and $\widetilde{A}$ intermingle. We shall discuss canonical filtrations of modules in Section 3.
Section 4 summarizes some preliminary results for the Gorensteinness of $\R(\q^d)$ to prove Theorem \ref{main}.

\section{The $(S_2)$-ification obtained by a sequence}
Let $A$ be a commutative ring and $M$ a nonzero $A$-module. We denote by $W$ the set of all non-zerodivisors on $M$. The localization of $M$ with respect to $W$ is denoted by $W^{-1}M$. The natural map $M\to W^{-1}M$ is injective, so that it allows us to view $M$ as an $A$-submodule of $W^{-1}M$. Fix a sequence $a, b\in A$ such that $a\in W$ and set
$$
\widetilde{M}=\{m/a\in W^{-1}M\mid m\in aM:_M b\}
$$ 
which is an $A$-submodule of $W^{-1}M$ containing $M$ with the equality $(a,b)\cdot\widetilde{M}/M=(0)$. 
The $A$-module $\widetilde{M}$ with the inclusion map $i:M\hookrightarrow \widetilde{M}$ 
has the following property given by \cite[Lemma (5.5)]{GY}.

\begin{lem}[\cite{GY}]\label{universal property}
Let $f: M \to X$ be a monomorphism of $A$-modules such that $(a,b)\cdot\Coker f=(0)$. Then there exists a unique homomorphism $g:X\rightarrow \widetilde{M}$ of $A$-modules such that $g\circ f=i$. Furthermore the following two assertions hold.
\begin{enumerate}
\item The homomorphism $g$ is injective if and only if $(0):_Xa=(0)$.
\item The homomorphism $g$ is bijective if $(0):_Xa=(0)$ and $aX:_Xb=aX$.
\end{enumerate}
\end{lem}

\begin{proof}
Take any element $x\in X$. 
Since $a\cdot\Coker f=(0)$, we have $a\cdot X\subseteq f(M)$, so that there exists a unique element $m\in M$ such that $ax=f(m)$, as $f$ is injective. 

If a homomorphism $g:X\to \widetilde{M}$ of $A$-modules satisfies the equality $g\circ f=i$, then $$ag(x)=g(ax)=g(f(m))=m,$$ so that $g(x)=m/a$. Hence, $g$ is uniquely determined.

For the existence, it is enough to show $m\in aM:_M b$. Since $b\cdot\Coker f=(0)$, there is an element $m'\in M$ such that $bx=f(m')$, so that we get the equalities $$f(am')=abx=bax=f(bm),$$ and hence $bm=am'\in aM$. Therefore, $m\in aM:_M b$. Thus we can find the $A$-homomorphism $g:X\to \widetilde{M}$ such that $g\circ f=i$.

It is routine to check the assertion 1. In order to prove the assertion 2, it is enough to show $g$ is surjective. Take any element $\alpha\in\widetilde{M}$. We write $\alpha=m_1/a$ for some $m_1\in aM:_M b$. Then $bm_1=am_2$ for some $m_2\in M$, so that $bf(m_1)=af(m_2)$, and hence $f(m_1)\in aX:_Xb=aX$.
We write $f(m_1)=ax'$ for some $x'\in X$, and then $g(x')=m_1/a=\alpha$. Thus $g$ is surjective.
\end{proof}

\

In the rest of this section, let $(A,\m)$ be a Noetherian local ring. Let $M$ be a finitely generated $A$-module, $\dim_A M\ge 2$, and $\depth_A M>0$. Suppose that $\H_\m^1(M)$ is a finitely generated $A$-module and set 
$$\mathfrak{c}=(0):_A \H^1_\m(M).$$ 
We assume that the sequence $a,b$ belongs to $\fkc$ and forms a filter regular sequence on $M$ (cf. \cite[p.252]{SV} for the definition).
Notice that such a sequence $a,b$ always exists, since $\H_\m^1(M)$ is a finitely generated $A$-module.

We shall show the sequence $a,b$ forms a regular sequence on $\widetilde{M}$. We obtain $$\H_\m^1(M)\cong \H_\m^0(M/aM)$$ as an $A$-module from applying the local cohomology functor $\H_\m^i(\ast )$ to the exact sequence
$
0\to M\xrightarrow{a} M\to M/aM\to 0
$
of $A$-modules induced by multiplication by the element $a$ on $M$ (recall that $a\cdot\H^1_\m(M)=(0)$). Since
$$
\H_\m^0(M/aM)=\left[\cup_{n>0}\ aM:_M \m^n\right]/aM
$$
and $b\cdot\H_\m^0(M/aM)=(0)$, we have 
$
\bigcup_{n>0}\ aM:_M \m^n\subseteq aM:_Mb
,$ 
and therefore $\bigcup_{n>0}\ aM:_M \m^n= aM:_Mb$ because the sequence $a,b$ forms a filter regular sequence on $M$. Hence, $$\H_\m^0(M/aM)=aM:_M b /aM.$$

By the definition of $\widetilde{M}$, the equality $aM:_M b=a\widetilde{M}$ holds, so that we see 
$$
aM:_M b /aM\cong \widetilde{M}/M
$$ 
as an $A$-module (recall that $a$ is a regular element on the both modules $M$ and $\widetilde{M}$). Thus we get 
$$
\H_\m^1(M)\cong\widetilde{M}/M
$$ 
as an $A$-module. Furthermore, we have the following.

\begin{prop}\label{G80}
$\depth_A \widetilde{M}\ge 2$ and $\H_\m^i(M)\cong\H_\m^i(\widetilde{M})$ as an $A$-module for all $i\ge 2$.
\end{prop}

\begin{proof}
We note $\depth_A \widetilde{M}>0$. Applying the local cohomology functor $\H_\m^i(\ast )$ to the natural exact sequence 
$$
0\to M\to \widetilde{M}\to \widetilde{M}/M\to 0
$$ 
of $A$-modules, we get the exact sequence
\begin{center}
$0\to \widetilde{M}/M\to \H_\m^1(M)\to \H_\m^1(\widetilde{M})\to 0$ and $0\to \H_\m^i(M)\to\H_\m^i(\widetilde{M})\to 0$
\end{center}
of $A$-modules for all $i\ge 2$. Hence, $\H_\m^1(\widetilde{M})=(0)$ because $\ell_A(\widetilde{M}/M)=\ell_A(\H_\m^1(M))<\infty$.
Therefore, $\depth \widetilde{M}\ge 2$ and $\H_\m^i(M)\cong\H_\m^i(\widetilde{M})$ as an $A$-module for all $i\ge 2$. 
\end{proof}

As consequence of the proposition above, let us state two corollaries as follows.
The next result due to \cite{G80}. 
\begin{cor}[\cite{G80}]\label{1980}
$\widetilde{M}$ is a Cohen-Macaulay $A$-module if and only if $\H^i_\m(M)=(0)$ for $i\neq 1, \dim M$.
\end{cor}

We recall that, for a positive integer $n$,
an $A$-module $N$ is said to satisfy the condition $(S_n)$ of Serre if $\depth N_\p\ge\min\{n, \dim N_\p\}$ for all $\p\in\Supp_A N$. The $A$-module $\widetilde{M}$ does not always satisfy the condition $(S_2)$ of Serre.

\begin{cor}\label{S2}The following two conditions are equivalent.
\begin{enumerate}
\item[$(1)$] $\widetilde{M}$ satisfies the condition $(S_2)$ of Serre.
\item[$(2)$] $\depth M_\p\ge\min\{2, \dim M_\p\}$ for all $\p\in\Supp_A M\setminus \{\m\}$.
\end{enumerate} 
\end{cor}

\begin{proof}
We have $\Supp_A M=\Supp_A \widetilde{M}$. For each $\p\in\Supp_A M\setminus \{\m\}$, we get $M_\p=\widetilde{M}_\p$ because $\widetilde{M}/M\cong\H_\m^1(M)$ which is finite length. Hence, the condition (2) is equivalent to saying that $\depth \widetilde{M}_\p\ge\min\{2, \dim \widetilde{M}_\p\}$ for all $\p\in\Supp_A \widetilde{M}$ because $\depth \widetilde{M}\ge 2$. This completes the proof.
\end{proof}

In particular, if $M_\p$ is a Cohen-Macaulay $A_\p$-module for all $\p\in\Supp_A M\setminus \{\m\}$, then $\widetilde{M}$ satisfies the condition $(S_2)$ of Serre. 

To see the sequence $a,b$ is a regular sequence on $\widetilde{M}$, we give the following.

\begin{lem}\label{2.5}
Let $a_1, a_2\in\m$. Then the following two conditions are equivalent.
\begin{enumerate}
	\item[$(1)$] The sequence $a_1, a_2$ is a regular sequence on $\widetilde{M}$.
	\item[$(2)$] The sequence $a_1, a_2$ forms a filter regular sequence on $M$.
\end{enumerate}
\end{lem}

\begin{proof}
In order to prove this equivalence, it is enough to show that $$\Ass_A\widetilde{M}/a_1\widetilde{M}=\Ass_AM/a_1M\setminus\{\fkm\}.$$ The natural exact sequence 
$
0\to M/a_1\widetilde{M}\cap M\to \widetilde{M}/a_1\widetilde{M}\to \widetilde{M}/a_1\widetilde{M}+M\to 0
$ 
of $A$-modules implies the equality $\Ass_AM/a_1\widetilde{M}\cap M=\Ass_A\widetilde{M}/a_1\widetilde{M}$, since $\fkm\not\in\Ass_A\widetilde{M}/a_1\widetilde{M}$ and $\Ass_A\widetilde{M}/a_1\widetilde{M}+M\subseteq \{\m\}$. Thus we get the required equality because 
$
a_1\widetilde{M}\cap M= \bigcup_{n>0}\ a_1M:_M \m^n.
$
\end{proof}

Thus the sequence $a,b$ forms a regular sequence on $\widetilde{M}$. Thanks to Lemma \ref{universal property}, the $A$-module $\widetilde{M}$ does not depend on the chosen filter regular sequence in $\fkc$, namely,

\begin{cor}
The equality $\widetilde{M}=\{m/a_1\in W^{-1}M\mid m\in a_1M:_M a_2\}$ holds for every filter regular sequence $a_1,a_2\in\fkc$ on $M$.
\end{cor}

Furthermore, we have the following.

\begin{prop}\label{structure}
Let $a_1$ and $a_2$ be elements in $\fkc$ such that the sequence $a_1,a_2$ is a filter regular sequence on $M$.
Let $f: M \to X$ be a monomorphism of $A$-modules such that $(a_1,a_2)\cdot\Coker f=(0)$. Then there exists a unique homomorphism $g:X\rightarrow \widetilde{M}$ of $A$-modules such that $g\circ f=i:M\hookrightarrow \widetilde{M}$ which is the inclusion map. Furthermore the following two assertions hold.
\begin{enumerate}
\item[$1.$] The following three conditions are equivalent.
\begin{enumerate}
	\item[$(1)$] The homomorphism $g$ is injective.
	\item[$(2)$] $a_1$ is a regular element on $X$.
	\item[$(3)$] $\depth X>0$ and $\ell_A(\Coker f)<\infty$.
\end{enumerate}
\vspace{2mm}
\item[$2.$] The following three conditions are equivalent.
\begin{enumerate}
	\item[$(1)$] The homomorphism $g$ is bijective.
	\item[$(2)$] $a_1,a_2$ is a regular sequence on $X$.
	\item[$(3)$] $\depth X\ge 2$ and $\ell_A(\Coker f)<\infty$.
\end{enumerate}
\end{enumerate}
\end{prop}

\begin{proof}
Thanks to Lemma \ref{universal property}, it suffices to show the assertions 1 and 2. Look at the commutative and exact diagram
$$
\begin{CD}
0   @>>>  M  @>f>>  X              @>>>  \Coker f         @>>>  0   \\
@.        @|        @VgVV                @VhVV                  @.  \\
0   @>>>  M  @>i>>  \widetilde{M}  @>>>  \widetilde{M}/M  @>>>  0   \\
\end{CD}
$$
of $A$-modules. For the assertion 1, the equivalence $(1)\Leftrightarrow (2)$ follows from \ref{universal property}. From the condition (1) we obtain that $\depth X>0$ (recall that $\depth \widetilde{M}\ge 2$) and that the induced morphism $h$ is injective and hence $\ell_A(\Coker f)<\infty$ because $\ell_A(\widetilde{M}/M)=\ell_A(\H^1_\m(M))<\infty$. The condition (3) implies that $\Ass_AM=\Ass_AX$ from the top exact sequence of the diagram above, so that $a_1$ is a regular element on $X$, as $a_1$ is a regular element on $M$.

For the assertion 2, it is enough to show the implication $(3)\Rightarrow (1)$. Assume that $\depth X\ge 2$ and $\ell_A(\Coker f)<\infty$. Then we obtain $\Coker f\cong \H_\m^1(M)$ as an $A$-module from applying the local cohomology functor $\H_\m^i(\ast )$ to the top exact sequence of the diagram above. From the assertion 1 we obtain $g$ is injective and hence so is $h$. Then the monomorphism $h$ must be bijective because $\Coker f$ and $\widetilde{M}/M$ have the same length $\ell_A(\H_\m^1(M))$. Therefore, $g$ is bijective.
\end{proof}

The proposition above means that the overmodule $\widetilde{M}$ of $M$ is characterized by two conditions $\depth_A \widetilde{M}\ge 2$ and $\ell_A(\widetilde{M}/M)<\infty$, namely,

\begin{cor}\label{minimality}
Let an $A$-module $X$ be an overmodule of $M$ satisfying $\depth_A X\ge 2$ and $\ell_A(X/M)<\infty$. Then $X=\widetilde{M}$. 
\end{cor}

\begin{proof}
We can find a sequence $a',b'\in (0):_AX/M$ that forms a filter regular sequence on $M$, since $\ell_A(X/M)<\infty$. Choose a positive integer $n$ such that ${a'}^n, {b'}^n\in \fkc$ and put $a_1={a'}^n$ and $a_2={b'}^n$. Then $a_1, a_2\in\fkc$ is a filter regular sequence on $M$ such that $(a_1, a_2)X/M=(0)$, and therefore $X=\widetilde{M}$ by Proposition \ref {structure}.
\end{proof}

Put $s=\dim M$. We simply say that a system $a_1,a_2,\dots ,a_s$ of parameters for $M$ is {\it standard}, if for all integers $n_1,n_2,\dots n_s\ge 1$, the sequence $a_1^{n_1},a_2^{n_2},\dots ,a_s^{n_s}$ forms a $d$-sequence on $M$ in any order. A standard system of parameters has good properties. See, e.g., \cite{STC,T} for details. For instance, there exists a standard system of parameters for $M$ if and only if $\H_\m^i(M)$ is a finitely generated $A$-module for all integers $i\neq s$. For a system $a_1,a_2,\dots ,a_s$ of parameters for $M$, $a_1,a_2,\dots ,a_s$ is standard if and only if $(a_1,a_2,\dots ,a_s)\cdot\H^i_\m(M/(a_1,a_2,\dots ,a_j)M)=(0)$ for $i,j\ge 0$ with $i+j<s$. 

Since a standard system of parameters for $M$ forms a filter regular sequence on  $M$ in any order, we obtain the following result from Lemma \ref{2.5}.

\begin{cor}
For any standard system $a_1,a_2,\dots ,a_s$ of parameters for $M$, every two elements $a_i, a_j\ (i\neq j)$ form a regular sequence on $\widetilde{M}$.
\end{cor}

\

In what follows, let $a_1,a_2,\dots ,a_s$ be a standard system of parameters for $M$. Then  $a_1,a_2,\dots ,a_s\in\fkc$ because $(a_1,a_2,\dots ,a_s)\cdot\H^1_\m(M)=(0)$.

\begin{prop}\label{independentold}
If $\alpha_1,\alpha_2,\dots , \alpha_s\in\widetilde{M}$ and $\sum_{i=1}^sa_i\alpha_i=0$, then $\alpha_1,\alpha_2,\dots , \alpha_s\in M$.
\end{prop}

\begin{proof}
For each $1\le i\le s$, we can write $\alpha_i=f_i/a_1$ for some $f_i\in a_1\widetilde{M}$. It is enough to show that $f_j\in a_1M$ for all integers $1\le j\le s$ if $\sum_{i=1}^sa_if_i=0$. To show this, we will use induction on $s$. 

Let $s=2$, and $a_1f_1+a_2f_2=0$. Then $f_1\in a_2M:_Ma_1=a_2\widetilde{M}$. Hence, $f_1\in a_1\widetilde{M}\cap a_2\widetilde{M}$. We have $$a_1\widetilde{M}\cap a_2\widetilde{M}=a_1M\cap a_2M$$ because $a_1\widetilde{M}\cap a_2\widetilde{M}=a_1a_2\widetilde{M}$ (recall that $a_1, a_2$ is a regular sequence on $\widetilde{M}$ and $a_i\widetilde{M}\subseteq M$ for all integers $1\le i\le s$). 
Therefore, we get $f_1\in a_1M$. Write $f_1=a_2m$ for some $m\in M$. Then $a_1a_2m+a_2f_2=0$, and hence $f_2\in a_1M$.

Assume that $s>2$. Let $1\le j\le s$. In order to show $f_j\in a_1M$, we may assume $j\neq s$. The sequence $a_1,a_2,\dots ,a_{s-1}$ forms a standard system of parameters for $M/a_s\widetilde{M}$ (see \cite[Proposition 2.2]{H82}). For each $1\le i\le s$, $f_i\in a_1\widetilde{M}$ and then $$f_j\ (\mod\ a_s\widetilde{M})\in a_1\cdot\widetilde{M}/a_s\widetilde{M}.$$ We have $\widetilde{M}/a_s\widetilde{M}\subseteq\widetilde{\widetilde{M}/a_s\widetilde{M}}=\widetilde{M/a_s\widetilde{M}}$ by Proposition \ref{structure}, so that $$f_i\ (\mod\ a_s\widetilde{M})\in a_1\cdot\widetilde{M/a_s\widetilde{M}}$$ for all $1\le i\le s-1$. From $\sum_{i=1}^{s-1}a_if_i\equiv 0$ $(\mod\ a_s\widetilde{M})$ we obtain $$f_j\in a_1M+a_s\widetilde{M}$$ for all $1\le i\le s-1$ by the inductive hypothesis. Since $f_j\in a_1\widetilde{M}$, we get $f_j\in a_1M+a_1\widetilde{M}\cap a_s\widetilde{M}$, and hence $f_j\in a_1M$ because $a_1\widetilde{M}\cap a_s\widetilde{M}=a_1a_s\widetilde{M}\subseteq a_1M$.
\end{proof}

For each $A$-module $N$ and each sequence $y_1,y_2,\dots ,y_\ell \in A$, we set 
$$
\widetilde{\Sigma}(y_1,y_2,\dots ,y_\ell; N)=\sum_{i=1}^\ell(y_1,\dots, y_{i-1}, y_{i+1},\dots y_\ell)N:_N y_i,
$$
which is an $A$-submodule of $N$. The proposition above yields the next.

\begin{cor}\label{independent}
$\widetilde{\Sigma}(a_1,a_2,\dots ,a_s; M)=\widetilde{\Sigma}(a_1,a_2,\dots ,a_s; \widetilde{M})$, and hence $a_1,a_2,\dots ,a_s$ is a standard system of parameters for $\widetilde{M}$.
\end{cor}

\begin{proof}
It is routine to check the equality $\widetilde{\Sigma}(a_1,a_2,\dots ,a_s; M)=\widetilde{\Sigma}(a_1,a_2,\dots ,a_s; \widetilde{M})$ by the lemma above. For the last assertion, see \cite[Appendix, Proposition 5]{KY}. 
\end{proof}

Corollary \ref{independent} implies the following result (notice that we assume $s=\dim M\ge 2$ in this paper).

\begin{cor}\label{independentCM}
The equality $\widetilde{\Sigma}(a_1,a_2,\dots ,a_s; M)=(a_1,a_2,\dots ,a_s)\widetilde{M}$ holds whenever the $A$-module $\widetilde{M}$ is Cohen-Macaulay.
\end{cor}

From now on, we consider the case where $M=A$. Notice that $\widetilde{A}$ is an overring of $A$ (cf. \cite{G80}). In what follows, we assume that $d=\dim A\ge 2$, $\depth A>0$, and $a_1,a_2,\dots ,a_d$ is a standard system of parameters for $A$. Then $\ell_A (\H^i_\m(A))<\infty$ for all integers $i$ and the ring $\widetilde{A}$ is just the ordinary $(S_2)$-ification of $A$ by Corollary \ref{S2} and Corollary \ref{minimality}. We put 
\begin{center}
$\fkq=(a_1,a_2,\dots ,a_d)A$, $\fka =\widetilde{\Sigma}(a_1,a_2,\dots ,a_d; A)$, and $\fkc=(0):_{A}\H^1_\m(A)$. 
\end{center}
Then $\fka$ is a common ideal of $A$ and $\widetilde{A}$ by Corollary \ref{independent}, and hence $\fka\subseteq \fkc$ (recall that $\fkc=A:\widetilde{A}$, since $\H^1_\m(A)\cong\widetilde{A}/A$ as an $A$-module).

\begin{prop}\label{proj}
Assume that $\H^i_\m(A)=(0)$ for $i\neq 1,d$, $\depth A=1$, and $\mathrm{r}_A(A)=1$. Then the following two conditions are equivalent.
\begin{enumerate}
	\item[$(1)$] $\fka =\fkc$.
	\item[$(2)$] $\e_\fkc (A)=2\ell_A (A/\fkc)$ and $\q$ is a reduction of $\fkc$.
\end{enumerate}
When this is the case, the equality $\fkc=\q\widetilde{A}$ holds and the ring $A/(a_1,a_2,\dots, a_{d-1})\widetilde{A}$ is Gorenstein.
\end{prop}

Before proving the proposition above, let us state the following two general lemmas.

\begin{lem}\label{r=1}
If $\depth A=1$ and $\mathrm{r}_A(A)=1$, then 
$\ell_A (\widetilde{A}/\fkc)=2\ell_A (A/\fkc)=2\ell_A (\widetilde{A}/A)$.
\end{lem}

\begin{proof}
Since $\mathrm{r}_A(A)=1$, we have $\mathrm{r}_A(\H_\m^1(A))=1$ because $\depth A=1$. Hence, $$[\H_\m^1(A)]^\vee\cong A/\fkc$$ as an $A$-module, where $[\quad]^\vee$ be the Matlis dual functor with respect to $A$.
Therefore, $\ell_A (A/\fkc)=\ell_A (\widetilde{A}/A)$, as $\H_\m^1(A)\cong \widetilde{A}/A$. Look at the natural exact sequence $$0\to A/\fkc\to \widetilde{A}/\fkc\to \widetilde{A}/A\to 0$$ of $A$-modules, and we see the equality $\ell_A (\widetilde{A}/\fkc)=2\ell_A (A/\fkc)$.
\end{proof}

\begin{lem}\label{d=0}Let $B$ be an Artinian local ring and $I$ an ideal of $B$. Let $[\quad]^\vee$ be the Matlis dual functor with respect to $B$. Then the following two conditions are equivalent.
\begin{enumerate}
	\item[$(1)$] $[I]^\vee\cong B/I$ as a $B$-module.
	\item[$(2)$] The natural exact sequence $0\to I\xrightarrow{i} B\xrightarrow{\epsilon} B/I\to 0$ is isomorphic to its dual sequence $0\to [B/I]^\vee\xrightarrow{\epsilon^\vee} [B]^\vee\xrightarrow{i^\vee} [I]^\vee\to 0$.
\end{enumerate}
When this is the case, the ring $B$ is Gorenstein.
\end{lem}
\begin{proof}
Let us only show the implication $(1)\Rightarrow (2)$. Let $h:B/I\to [I]^\vee$ be an isomorphism given by the condition $(1)$. Take an element $x\in [B]^\vee$ such that $i^\vee(x)=h(1)$. Then we define an $B$-linear map $g:B\to [B]^\vee$ carrying $1_B$ to $x$. Let $f$ be the composition of the natural map $\theta :I\xrightarrow{\sim}[I]^{\vee\vee}$ and the dual map $h^\vee:[I]^{\vee\vee} \xrightarrow{{\sim}}[B/I]^\vee$ of $h$. Then $f$ is an isomorphism of $B$-modules. It is routine to check the following diagram is commutative. 
$$
\begin{CD}
0@>>> I@>i>> B@>\epsilon>> B/I@>>> 0\\
@.@Vf V\wr V@Vg VV@Vh V\wr V\\
0@>>> [B/I]^\vee@>\epsilon^\vee>> [B]^\vee@>i^\vee>> [I]^\vee@>>> 0.
\end{CD}
$$ 
Therefore, we have $g$ is an isomorphism of $B$-modules.
\end{proof}

\begin{proof}[Proof of Proposition \ref{proj}]
The ring $\widetilde{A}$ is a Cohen-Macaulay $A$-module by Proposition \ref{G80}, so that 
$\fka=\q\widetilde{A}$ by Corollary \ref{independentCM}, and then we see $\q$ is a reduction of $\fka$.
Since $\e_\q (A)=\e_\q (\widetilde{A})$, we get $\e_\q (A)=\ell_A (\widetilde{A}/\q\widetilde{A})$.

$(1)\Rightarrow (2)$. Since $\fkc=\fka$, the ideal $\q$ is a reduction of $\fkc$, and then $\e_\fkc (A)=\e_\q (A)=\ell_A (\widetilde{A}/\q\widetilde{A})$. Since $\q\widetilde{A}=\fkc$, we get $\e_\fkc (A)=2\ell_A (A/\fkc)$ by Lemma \ref{r=1}.

$(2)\Rightarrow (1)$. Since $\e_\q (A)=2\ell_A (A/\fkc)$, we get $\ell_A (\widetilde{A}/\q\widetilde{A})=\ell_A (\widetilde{A}/\fkc)$ by Lemma \ref{r=1}. Therefore, $\fka=\fkc$ because $\fka=\q\widetilde{A}$.

For the last assertion, the equality $\fkc=\q\widetilde{A}$ directly follows from the condition (1). Put $\q_{d-1}=(a_1,a_2,\dots, a_{d-1})A$ and $B=A/\q_{d-1}\widetilde{A}+a_dA$. In order to prove the ring $A/\q_{d-1}\widetilde{A}$ is Gorenstein, it is enough to show that so is the ring $B$, as $a_d$ is a regular element on $A/\q_{d-1}\widetilde{A}$.
Let $I$ be the kernel of the natural surjection 
$
B\to A/\q\widetilde{A}
$
of rings. Then 
$$
I=\q\widetilde{A}/\q_{d-1}\widetilde{A}+a_dA\cong a_d\widetilde{A}/\q_{d-1}\widetilde{A}\cap a_d\widetilde{A}+a_dA\cong a_d\widetilde{A}/a_dA\cong \widetilde{A}/A\cong \H_\m^1(A)
$$ 
as an $A$-module because $\q_{d-1}\widetilde{A}\cap a_d\widetilde{A}=a_d\q_{d-1}\widetilde{A}\subseteq a_dA$. 

Let $[\quad]^\vee$ be the Matlis dual functor with respect to $A$. Since $\mathrm{r}_A(\H_\m^1(A))=1$, we have $[\H_\m^1(A)]^\vee\cong A/\fkc$ as an $A$-module, so that $[I]^\vee\cong A/\fkc$ as an $A$-module. 
Since $\fkc=\q\widetilde{A}$, we get $[I]^\vee\cong A/\q\widetilde{A}$ as an $A$-module. Thanks to Lemma \ref{d=0}, the ring $B$ is Gorenstein.
\end{proof}

Proposition \ref{proj} yields the next result.

\begin{thm}\label{gor}
Assume that $\H^i_\m(A)=(0)$ for $i\neq 1,d$ and $\ell_A (\H^1_\m(A))<\infty$. Then the ring $\widetilde{A}$ is Gorenstein if $\depth A=1$, $\mathrm{r}_A(A)=1$, and $\e_\fkc (A)=2\ell_A (A/\fkc)$.\end{thm}

\begin{proof}
We may assume the field $A/\m$ is infinite. Then we can find a standard system $b_1,b_2,\dots , b_d$ of parameters for $A$ such that $(b_1,b_2,\dots , b_d)A$ is a reduction of the ideal $\fkc$ (cf. \cite[Corollary 3.7]{T}). We put $\q_{d-1}=(b_1, b_2, \dots , b_{d-1})A$ and $\q_{d}=(b_1, b_2, \dots , b_{d})A$. Then the ring $A/\q_{d-1}\widetilde{A}$ is a Gorenstein local ring by Proposition \ref{proj}. 
Taking the $A/\q_{d-1}\widetilde{A}
$-dual of the natural exact sequence 
$$
0\to A/\q_{d-1}\widetilde{A} \to \widetilde{A}/\q_{d-1}\widetilde{A}\to \widetilde{A}/A\to 0
$$ 
of $A$-modules, we get the exact sequence 
$$
0\to\K_{\widetilde{A}/\q_{d-1}\widetilde{A}} \to A/\q_{d-1}\widetilde{A}\to [\widetilde{A}/A]^\vee\to 0
$$ 
of $A$-modules, where $[\quad]^\vee$ is the Matlis dual functor with respect to $A$ and $\K_{\widetilde{A}/\q_{d-1}\widetilde{A}}$ is the canonical module of the ring $\widetilde{A}/\q_{d-1}\widetilde{A}$. Since $\mathrm{r}_A(A)=1$, $[\widetilde{A}/A]^\vee\cong A/\fkc$ as an $A$-module, so that we get an exact sequence
$$
0\to\K_{\widetilde{A}/\q_{d-1}\widetilde{A}} \to A/\q_{d-1}\widetilde{A}\to A/\fkc\to 0
$$
of $A$-modules. This is isomorphic to the natural exact sequence 
$$
0\to\q_{d}\widetilde{A}/\q_{d-1}\widetilde{A} \to A/\q_{d-1}\widetilde{A}\to A/\q_{d}\widetilde{A}\to 0
$$
of $A$-modules, as $\fkc = \q_{d}\widetilde{A}$ by Proposition \ref{proj}. We have 
$$
\q_{d}\widetilde{A}/\q_{d-1}\widetilde{A}\cong b_d\widetilde{A}/\q_{d-1}\widetilde{A}\cap b_d\widetilde{A} \cong\widetilde{A}/\q_{d-1}\widetilde{A}
$$ 
as an $A$-module, since $\q_{d-1}\widetilde{A}\cap b_d\widetilde{A}= b_d\q_{d-1}\widetilde{A}$. Therefore, $\K_{\widetilde{A}/\q_{d-1}\widetilde{A}}\cong \widetilde{A}/\q_{d-1}\widetilde{A}$ as an $A$-module. This is an isomorphism of $\widetilde{A}$-modules, and thus the ring $\widetilde{A}$ is Gorenstein.
\end{proof}

We set $R=\R(\q)$ and $S=\R(\q\widetilde{A})$. 
Let $A[X_1,X_2,\dots X_d]$ be the $\Z$-graded polynomial ring over $A$ such that $\deg X_i=1$ for all $1\le i\le d$ and $\deg a=0$ for all $a\in A$. Let $\varphi :A[X_1,X_2,\dots X_d]\to R$ be the graded homomorphism of $A$-algebras such that $\varphi (X_i)=a_it$ for all $1\le i\le d$. 
With this notation, we will close this section with a consequence of Proposition \ref{independentold} as given below.

\begin{cor}\label{independent2}
$S/R\cong A[X_1,X_2,\dots X_d]\otimes_A\widetilde{A}/A$ as a graded $A[X_1,X_2,\dots X_d]$-module.
\end{cor}

\begin{proof}
Let $\widetilde{\varphi} :\widetilde{A}[X_1,X_2,\dots X_d]\to S$ be the graded homomorphism of $\widetilde{A}$-algebras such that $\varphi (X_i)=a_it$ for all $1\le i\le n$. Then we have the commutative diagram
$$
\begin{CD}
R@>>>S\\
@A\varphi AA@A\widetilde{\varphi} AA\\
A[X_1,X_2,\dots X_d]@>>>\widetilde{A}[X_1,X_2,\dots X_d]
\end{CD}
$$
of $A[X_1,X_2,\dots X_d]$-modules, where the horizontal homomorphisms are inclusion maps. To get the required isomorphism, it is enough to prove $\Ker \varphi \supseteq\Ker \widetilde{\varphi}$. According to \cite{H80} and \cite{V}, $\Ker \widetilde{\varphi}$ is generated by linear forms, since the sequence $a_1,a_2,\dots ,a_d$ is a $d$-sequence on $\widetilde{A}$ by Corollary \ref{independent}. All coefficients of these linear forms are contained in $A$ by Proposition \ref{independentold}. Therefore, $\Ker \varphi \supseteq\Ker \widetilde{\varphi}$.
\end{proof}

\section{Canonical filtrations}

The purpose of this section is to prepare some basic result of canonical filtrations of Rees modules that we need to prove Theorem \ref{main}.

Let $(A, \m)$ be a Noetherian local ring, $I$ an ideal of $A$. Let $\fkM$ stand for the graded maximal ideal of $\R (I)$. We will make use of the standard notation for graded modules as follows.
Let $E$ be a finitely generated graded $\R (I)$-module and let $n\in\mathbb{Z}$.  We denote by $\H_\fkM^n(E)$ the $n^{\th}$ graded local cohomology module of $E$ with respect to $\fkM$. Let $E_n$ stand for the $n^{\th}$ homogeneous component of $E$. Put $E_{\ge n}=\oplus_{m\ge n}E_m$ and $E_+=E_{\ge 1}$. Let $E(n)$ denote the graded $\R (I)$-module whose grading is given by $[E(n)]_m=E_{n+m}$ for all $m\in\mathbb{Z}$. 
For each integer $i$, set $$\mathrm{a}_i (E) = \max \{ n \in \Z \mid [\H^i_\fkM(E)]_n\neq (0) \}.$$ When $s=\dim E$, we denote $\mathrm{a}_s (E)$ simply by $\mathrm{a} (E)$ that is called the $a$-invariant of $E$.

Let $M$ be a finitely generated $A$-module and $\M=\{M_n\}_{n\in\Z}$ stand for a family of $A$-submodules of $M$. Assume that $\M$ is an $I$-filtration of $M$, namely, it satisfies the following three conditions.
\begin{enumerate}
	\item [(i)] $M_{n+1} \subseteq M_n$ for all $n\in \Z$.
	\item [(ii)] $IM_n \subseteq M_{n+1}$ for all $n \in \Z$.
	\item [(iii)] $M_0=M$.
\end{enumerate}
Let $t$ be an indeterminate over $A$.
We define
\begin{enumerate}
	\item[] $\R(\M) = \sum_{n\ge 0}\{t^n\otimes x\mid x\in M_n\}\subseteq A[t]\otimes_A M$ and
	\item[] $\R'(\M) = \sum_{n\in\Z}\{t^n\otimes x\mid x\in M_n\}\subseteq A[t,t^{-1}]\otimes_A M$,
\end{enumerate}
and call them the Rees module of $\M$ and the extended Rees module of $\M$. They are respectively a graded $\R(I)$-submodule of $A[t]\otimes_A M$ and a graded $\R'(I)$-submodule of $A[t,t^{-1}]\otimes_A M$. The graduation of the Rees module (resp. the extended Rees module) is given by  $[\R(\M)]_n =\{t^n\otimes x\mid x\in M_n\}$ if $n\ge 0$ or $[\R(\M)]_n =(0)$ if $n< 0$. (resp. $[\R'(\M)]_n =\{t^n\otimes x\mid x\in M_n\}$ for each $n\in\Z$). 
Let 
$$
\G(\M) = \R'(\M)/t^{-1}\R'(\M)
$$ 
and call it the associated graded module of $\M$.
In the case where $M=A$ and $\M=\{I^n\}_{n\in\Z}$, we set $\R'(I)=\R'(\M)$ and  $\G (I)=\G (\M)$ for simplicity, which are called the associated graded ring of $I$ and the extended Rees algebra of $I$, respectively.

We consider the following exact sequences
$$
\begin{CD}
(\sharp_1)\qquad @.0@>>> \R(\M)_+@>i>>\R(\M)@>p>> M@>>> 0@.\quad \mathrm{and} \\
(\sharp_2)\qquad @.0@>>> \R(\M)_+(1)@>j>>\R(\M)@>\epsilon>> \G(\M)@>>> 0@.\\
\end{CD}
$$
of graded $\R(I)$-modules, where $i$ is the inclusion map,
$p$ is the canonical epimorphism, and the sequence $(\sharp_2)$ is the nonnegative part of the natural exact sequence $$0\to \R'(\M)(1)\xrightarrow{t^{-1}} \R'(\M)\to \G(\M)\to 0$$ of graded $\R'(I)$-modules induced by multiplication by $t^{-1}$ on $\R'(\M)$. 

We begin with the following lemma which is known, but let us give a brief proof for the sake of completeness.

\begin{lem}\label{ai}
Let $i$ be an integer. Assume that $\R(\M)$ is a finitely generated $\R(I)$-module. Then $\mathrm{a}_i(\R(\M))<0$ whenever $\mathrm{a}_i(\G(\M))<0$. 
Furthermore, $\mathrm{a}(\R(\M))=-1$ if $\dim \R(\M)=\dim \G(\M)+1$.
\end{lem}

\begin{proof}
Let $i$ be an integer. The sequences $(\sharp_1)$ and $(\sharp_2)$ above induce two exact sequences 
\begin{enumerate}
	\item[(a)] $\H_\fkM^{i-1}(M)\to\H_\fkM^i(\R(\M)_+)\to\H_\fkM^i(\R(\M))\to\H_\fkM^i(M)$ and 
	\item[(b)] $\H_\fkM^i(\R(\M)_+)(1)\to\H_\fkM^i(\R(\M))\to\H_\fkM^i(\G(\M))$ 
\end{enumerate}
of graded local cohomology modules. For each integer $m\neq 0$, it follows from (a) that $[\H_\fkM^i(\R(\M)_+)]_m\cong [\H_\fkM^i(\R(\M))]_m$ as an $A$-module because $\H_\fkM^{i-1}(M)$ and $\H_\fkM^i(M)$ are concentrated in degree $0$. Then for each integer $n\neq -1$, from (b) we obtain an $A$-homomorphism $f_n:[\H_\fkM^i(\R(\M))]_{n+1}\to[\H_\fkM^i(\R(\M))]_n$. When $\mathrm{a}_i(\G(\M))<0$, the homomorphism $f_n$ is surjective for each integer $n\ge 0$, and therefore $\mathrm{a}_i(\R(\M))<0$, as $[\H_\fkM^i(\R(\M))]_{n}=(0)$ for all integers $n\gg 0$.

For the last assertion, we consider the case where $i=\dim \R(\M)=\dim \G(\M)+1$. Then $\mathrm{a}_i(\G(\M))=-\infty$, and hence $\mathrm{a}_i(\R(\M))< 0$. Suppose that $[\H_\fkM^i(\R(\M))]_{-1}= (0)$. Then we see $\mathrm{a}_i(\R(\M))=-\infty$, since the homomorphism $f_n$ is surjective for each integer $n\le -2$. But this is impossible, as $i=\dim \R(\M)$, and thus we get $\mathrm{a}(\R(\M))=-1$.
\end{proof}

Let $\ell$ be a positive integer and let $I$ be generated by elements $a_1,a_2,\dots ,a_\ell\in A$. In the rest of this section, we consider the case where $I\neq A$ and 
$$
\M=\{I^nM\}_{n\in\Z}.
$$ 
We assume that there exists an $M$-regular element in $I$. Then we may choose $a_1$ is a regular element on $M$. Put $s=\dim M$. Then $\dim\R'(\M)=\dim\R(\M)=s+1$ and $\dim\G(\M)=s$ by \cite[Theorem 4.5.6 and Theorem 4.5.11]{BH}.

In what follows, we assume that the ring $A$ is a homomorphic image of a Gorenstein local ring $C$. Let $\psi : C\to A$ stand for the epimorphism of rings and let $$\P=C[X_1,X_2,\dots X_\ell]$$ be the $\Z$-graded polynomial ring over $C$ such that $\deg X_i=1$ for all $1\le i\le \ell$ and $\deg c=0$ for all $c\in C$. We set $\K_{\P}=\P(-\ell)$ that is the graded canonical module of $\P$. Let $\varphi :\P\to \R(I)$ be the graded homomorphism of $C$-algebras such that $\varphi (X_i)=a_it$ for all $1\le i\le \ell$ and $\varphi (c)=\psi (c)$ for all $c\in C$. 
We denote the Matlis dual functor with respect to $C$ by $[\quad]^\vee$. 

Put $m=\dim C+\ell-s-1$ and $L=\Ext_{\P}^m(\R(\M)_+,\K_{\P})$. We set 
\begin{align*}
&\K_M=\Ext_{\P}^{m+1}(M,\K_{\P}),\\ 
&\K_{\R(\M)}=\Ext_{\P}^m(\R(\M),\K_{\P}),\ and \\
&\K_{\G(\M)}=\Ext_{\P}^{m+1}(\G(\M),\K_{\P}).
\end{align*}
They are finitely generated graded $\R(I)$-modules. Notice that $\K_M$ is concentrated in degree $0$.
Fix an $A$-module $X$ such that $$X\cong\K_M$$ as an $A$-module. Taking $\K_\P$-dual of the exact sequences $(\sharp_1)$ and $(\sharp_2)$, we get two exact sequences 
$$
\begin{CD}
(\flat_1 )\qquad @. 0   @>>>\K_{\R(\M)}  @>i^\ast>> L     @>>>X       @>>>  \Ext_{\P}^{m+1}(\R(\M),\K_{\P})  \\
(\flat_2 )\qquad @. 0   @>>>\K_{\R(\M)}  @>j^\ast>> L(-1) @>>>\K_{\G(\M)} @>>>  \Ext_{\P}^{m+1}(\R(\M),\K_{\P})
\end{CD}
$$
of graded $\R(I)$-modules. Since $\mathrm{a}(\R(\M))=-1$ (see Lemma \ref{ai}), we have $\K_{\R(\M)}=[\K_{\R(\M)}]_{\ge 1}$ by the graded local duality theorem. Since the graded $\R(I)$-module $X$ is concentrated in degree $0$, we get the isomorphism $i^\ast_+:\K_{\R(\M)}\xrightarrow{\sim} L_+$ of graded $\R(I)$-modules which is defined by $i^\ast_+(\xi)=i^\ast(\xi)$ for all $\xi\in\K_{\R(\M)}$. Let 
$$
\beta:L_+\rightarrow L(-1)
$$ 
be composition of ${i^\ast_+}^{-1}$ and $j^\ast$. Then $\beta$ is a monomorphism of graded $\R(I)$-modules. 

Put $T=\{X_1^n\mid n\ge 0\}$, which is a multiplicatively closed subset of $\P$. Let 
$$
h:\R'(I)\to T^{-1}\R'(I)
$$ 
be the natural map of the localization of $\R'(I)$ with respect to $T$. Put $$u=h(t^{-1})$$ and then $u\in T^{-1}\R(I)$, as $T^{-1}\R(I)=T^{-1}\R'(I)$. Since $T^{-1}(L_+)=T^{-1}L$, we have the commutative diagram
$$
\begin{CD}
T^{-1}L@>u>>T^{-1}L(-1)\\
@AAA   @AAA       \\
L_+    @>\beta>>      L(-1)\\
\end{CD}
$$
of graded $\R(I)$-modules, where the vertical homomorphisms are the natural ones, which are injective (recall that $a_1$ is a regular element on $M$). Thus it allows us to view $\beta$ as a restriction of multiplication by $u~ (=t^{-1})$ on $T^{-1}L$.

We obtain from ($\flat_1$) the monomorphism $\alpha:L_0\hookrightarrow X$, as $\mathrm{a}(\R(\M))=-1$. We set $$K=\im \alpha$$ and
$$
\k_n=\left\{
\begin{array}{ll}
\alpha(\underbrace{\beta(\cdots\beta(\beta}_{n~ \mathrm{times}}(L_n))\cdots))&\mbox{ if $n\ge 0$}\cr
K&\mbox{ if $n<0$},\cr
\end{array}
\right.
$$
which are $A$-submodules of $K$. Put $\k=\{\k_n\}_{n\in\Z}$. Then we have

\begin{lem}\label{k}
$\k$ is an $I$-filtration of $K$ such that $L\cong \R(\k)$ as an $\R(I)$-module.
\end{lem}

\begin{proof}
Let $n$ be a nonnegative integer. Since $\beta(L_{n+1})\subseteq L_n$, we have $\k_{n+1}\subseteq\k_n$.
Since the homomorphism $\beta$ is the restriction of multiplication by $u\ (=h(t^{-1}))$ on $T^{-1}L$, we see 
$$I\k_n=\alpha(\underbrace{\beta(\cdots\beta(\beta}_{n~ \mathrm{times}}(IL_n))\cdots))=\alpha(\underbrace{\beta(\cdots\beta(\beta}_{n+1~ \mathrm{times}}(It\cdot L_n))\cdots)),$$
so that $I\k_n\subseteq \k_{n+1}$, as $It\cdot L_n\subseteq L_{n+1}$. Thus $\k$ is an $I$-filtration of $K$.

Let us show $L\cong\R(\k)$ as a graded $\R(I)$-module.
For each integer $n\ge 0$, let $\gamma_n:[L]_n\to \k_n$ be a homomorphism of $A$-modules defined by
$$
\gamma_n(x)=\alpha(\underbrace{\beta(\cdots\beta(\beta}_{n~ \mathrm{times}}(x))\cdots))
$$
for all $x\in [L]_n$ and let $\rho :L\to \R(\k)$ be a homomorphism of graded $\R(I)$-modules carrying $x\in [L]_n$ to $t^n\otimes \gamma_n (x)\in [\R(\k)]_n$. Then we see $\rho$ is an isomorphism of graded $\R(I)$-modules by the definition of $\k$. 
\end{proof}

Let us refer to an $I$-filtration $\w=\{\w_n\}_{n\in\Z}$ of an $A$-module $W$ with a graded $\R(I)$-isomorphism $L\xrightarrow{\sim}\R(\w)$ as {\it a canonical $I$-filtration of $W$ with respect to $M$}. Then the $I$-filtration $\k$ is a canonical $I$-filtration of $K$ with respect to $M$ by Lemma \ref{k}.

We will pick up some basic lemmas needed later on canonical filtrations as follows.

\begin{lem}\label{canonical filtration}
Let $W$ be an $A$-module and $\w=\{\w_n\}_{n\in\Z}$ a canonical $I$-filtration of $W$ with respect to $M$. Then the following two assertions hold.
\begin{enumerate}
	\item[$(1)$] $\K_{\R(\M)}\cong \R(\w)_+$ as a graded $\R(I)$-module.
	\item[$(2)$] There exists an exact sequence $0\to\G(\w)(-1)\to\K_{\G(\M)}\to\Ext_{\P}^{m+1}(\R(\M),\K_{\P})$ of graded $\R(I)$-modules.
\end{enumerate}
\end{lem}
\begin{proof}
Let $\varphi :L\xrightarrow{\sim} \R(\w)$ be the given isomorphism of graded $\R(I)$-modules.
Then we obtain an isomorphism $\K_{\R(\M)}\xrightarrow{\sim} \R(\w)_+$ of graded $\R(I)$-modules from composition of the isomorphisms $i^\ast_+:\K_{\R(\M)}\xrightarrow{\sim} L_+$ and $\varphi_+ :L_+\xrightarrow{\sim} \R(\w)_+$ of graded $\R(I)$-modules, where $i^\ast_+$ and $\varphi_+$ are the positive part of $i^\ast$ and $\varphi$.

Put $\psi =\varphi (-1)\circ \beta\circ {\varphi_+}^{-1}$. Look at the commutative and exact diagram
$$
\begin{CD}
0@>>>L_+    @>\beta>>      L(-1)@>>> \K_{\G(\M)} @>>>   \Ext_{\P}^{m+1}(\R(\M),\K_{\P})\\
@.@V\varphi_+V\wr V   @V\varphi(-1) V\wr V       \\
@.\R(\w)_+@>\psi >>\R(\w)(-1)\\
\end{CD}
$$
of graded $\R(I)$-modules. We recall that $\beta$ is the restriction of multiplication by $u$ on $T^{-1}L$. Then we see $\psi $ is a restriction of multiplication by $u$ on $T^{-1}\R(\w)$, so that $\Coker\psi =\G(\w)(-1)$. Thus we get the required exact sequence $0\to\G(\w)(-1)\to\K_{\G(\M)}\to\Ext_{\P}^{m+1}(\R(\M),\K_{\P})$ of graded $\R(I)$-modules.
\end{proof}

\begin{prop}\label{full canonical filtration}
Assume $\mathrm{a}(\G(\M))<0$. Then $K=X$, and hence there exists a canonical $I$-filtration of $\K_M$ with respect to $M$.
\end{prop}
\begin{proof}
Since $\mathrm{a}(\G(\M))<0$, we have $\mathrm{a}_s(\R(\M))<0$ by Lemma \ref{ai}. By the graded local duality theorem, $$\widehat{C}\otimes_C\Ext_{\P}^{m+1}(\R(\M),\K_{\P})\cong[\H_\fkM^s(\R(\M))]^\vee$$ as a graded $\P$-module, and hence $[\Ext_{\P}^{m+1}(\R(\M),\K_{\P})]_0=(0)$. Therefore, $\alpha$ is surjective, so that $K=X$. For the last assertion, recall that $X\cong\K_M$ as an $A$-module.
\end{proof}

\begin{lem}\label{ll}
Let $W$ be a finitely generated $A$-module and $\w=\{\w_n\}_{n\in\Z}$ a canonical $I$-filtration of $W$ with respect to $M$. Then $\w_n=W$ for all integers $n\le -\mathrm{a}(\G(\M))-1$.
\end{lem}
\begin{proof}
Put $a=\mathrm{a}(\G(\M))$.
By the graded local duality theorem, $$\widehat{C}\otimes_C\Ext_\P^{m+1}(\G(\M),\K_\P)\cong[\H_\fkM^s(\G(\M))]^\vee$$ as a graded $\P$-module, so that,  for all integers $n\le -a-1$, $[\K_{\G(\M)}]_{n}=(0)$ and then $[{\G(\w)}]_{n-1}\ (\cong\w_{n-1}/\w_{n})=(0)$ by Lemma \ref{canonical filtration}. Hence, $\w_{-a-1}=\w_{-a-2}=\cdots$. Therefore, we get $\w_n=W$ for all $n\le -a-1$, since $\w_i=W$ for all $i\le 0$.
\end{proof}

\begin{lem}\label{aG}
$\mathrm{a}(\G(\M))= -s$ whenever the finitely generated $A$-module $M$ is an overring of $A$ and the ideal $I$ is generated by a system of parameters for $M$.
\end{lem}
\begin{proof}
Note that $s=\dim M=\dim A$ and that $\widehat{A}\otimes_A M$ is the direct product of local rings of dimension at most $s$. Then we see $\widehat{A}\otimes_A\G(\M)$ is the direct product of associated graded rings of dimension at most $s$. Thus it suffices to consider the case where $M=A$. Put $G=\G(\M)$. Then $\mathrm{a}(G)=\mathrm{a}(G/\m G)$ (recall that $A/\m\otimes_A\H_\fkM^s(G)\cong \H_\fkM^s(G/\m G)$ and that $[\H_\fkM^s(G)]_i$ is a finitely generated $A$-module for all $i\in\Z$).
Since $G/\m G$ is the graded polynomial ring in $s$ variables of degree $1$ over the field $A/\m$, we have $\mathrm{a}(G)= -s$.
\end{proof}

Let $M'$ be a finitely generated $A$-module of dimension $s$ such that the $A$-module $M$ is an $A$-submodule of $M'$. Assume that there exists a regular element on $M'$ in the ideal $I$. We set $\M'=\{I^nM'\}_{n\in\Z}$ and put $L'=\Ext_{\P}^m(\R(\M')_+,\K_{\P})$. Then we have the following.

\begin{lem}\label{two canonical filtrations}
Assume that $s\ge 2$, $\dim M'/M\le s-2$, and $\mathrm{a}(\G(\M'))<0$. Then there exists a canonical $I$-filtration $\w=\{\w_n\}_{n\in\Z}$ of $X$ with respect to $M'$ such that $\w_n\subseteq \k_n$ for all $n\in\Z$ and that a diagram
$$
\begin{CD}
L'@>>>L\\
@V\rho_1 V\wr V     @V\rho_2 V\wr V \\
\R(\w)@>>>\R(\k)\\
\end{CD}
$$
of graded $\R(I)$-modules is commutative for some isomorphisms $\rho_1$ and $\rho_2$, where the top homomorphism is induced by the inclusion map $\R(\M)_+\hookrightarrow \R(\M')_+$, and the bottom is the inclusion map. Hence, one has a commutative diagram
$$
\begin{CD}
\K_{\R(\M')}@>>>\K_{\R(\M)}\\
@VV\wr V     @VV\wr V \\
\R(\w)_+@>>>\R(\k)_+\\
\end{CD}
$$
of graded $\R(I)$-modules.
\end{lem}
\begin{proof}
Since $s\ge 2$ and $\dim M'/M\le s-2$, $\K_{M'}\cong\K_{M}\cong X$ as an $A$-module. We have two commutative and exact diagrams
$$
\begin{array}{l}
(\sharp_1')\qquad
\begin{CD}
0@>>> \R(\M)_+@>i>>\R(\M)@>p>> M@>>> 0\\
@.    @VVV          @VVV       @VVV   \\
0@>>> \R(\M')_+@>i'>>\R(\M')@>p'>> M'@>>> 0\\
\end{CD}
\\
\hspace{-1.6cm}\mathrm{and}\\
(\sharp_2')\qquad
\begin{CD}
0@>>> \R(\M)_+(1)@>j>>\R(\M)@>\epsilon >> \G(\M)@>>> 0\\
@.    @VVV          @VVV       @VVV   \\
0@>>> \R(\M')_+(1)@>j'>>\R(\M')@>\epsilon ' >> \G(\M')@>>> 0\\
\end{CD}
\end{array}
$$
of graded $\R(I)$-modules, where the vertical maps are the natural ones. In the same way as the construction of $\k$, the diagrams above induce two commutative diagrams
\begin{center}
($\natural_1$)\qquad
$
\begin{CD}
L_+@>\beta  >>L(-1)\\
@AAA          @AAA \\
L'_+@>\beta '>>L'(-1)\\
\end{CD}
$
\qquad and\qquad ($\natural_2$) \qquad 
$
\begin{CD}
L_0@>\alpha >>X\\
@AAA          @A\pi A\wr A \\
L'_0@>\alpha '>>X\\
\end{CD}
$
\end{center}
of graded $\R(I)$-modules.
Since $\mathrm{a}(\G(\M'))<0$, we have $\alpha '$ is an isomorphism by Proposition \ref{full canonical filtration}, and therefore we see every homomorphism in the diagram ($\natural_2$) is bijective. Hence, $X=\im\alpha=\im\alpha'$. We put
$$
\k'_n=\left\{
\begin{array}{ll}
\alpha'(\underbrace{\beta'(\cdots\beta'(\beta'}_{n~ \mathrm{times}}(L'_n))\cdots))&\mbox{ if $n>0$}\cr
X&\mbox{ if $n\le0$}.\cr
\end{array}
\right.
$$
Then $\k'=\{\k'_n\}_{n\in\Z}$ is a canonical $I$-filtration  of $X$ with respect $M'$ by Lemma \ref{k}.
Because of commutativity of diagrams ($\natural_1$) and ($\natural_2$), we get the inclusions
$\pi(\k'_n)\subseteq \k_n$ for all $n\in\Z$, and then we have the commutative diagram
\begin{center}
$
\begin{CD}
L'@>>>L\\
@V\rho' V\wr V          @V\rho V\wr V \\
R(\k')@>\eta >>\R(\k)\\
\end{CD}
$
\end{center}
of graded $\R(I)$-modules, where the top homomorphism is induced by the inclusion map $\R(\M)_+\hookrightarrow\R(\M')_+$, $\eta$ is defined by $t^n\otimes x \mapsto t^n\otimes \pi (x)$, and the vertical ones are the isomorphisms obtained by Lemma \ref{k}.
We put $\w_n=\pi(\k'_n)$ and then $\w_n\subseteq \k_n$ for all $n\in\Z$. We recall $\pi$ is an isomorphism of $A$-modules, so that $\w=\{\w_n\}_{n\in\Z}$ must be a canonical $I$-filtration of $X$ with respect to $M'$, as so is $\k'$. Then from the diagram above and the equality $\R(\w)=\im \eta$, we obtain the required commutative diagram. 

By Lemma \ref{canonical filtration}, $\K_{\R(\M)}\cong L_+$ and $\K_{\R(\M')}\cong L'_+$ as graded $\R(I)$-module, which implies the last required commutative diagram.
\end{proof}

In the rest of this section, we consider the case where the ideal $I$ is a standard parameter ideal. For each graded $\R(I)$-module $H$, we set $\Soc H=(0):_H\fkM$.

\begin{prop}\label{concent}
Assume that the ideal $I$ is generated by a standard system of parameters for $M$. Then $\K_{\G(\M)}$ is generated by homogeneous elements of degrees $s$ and $s+1$ as a graded $\R(I)$-module.
\end{prop}

\begin{proof}
It is enough to prove that $\Soc \H_\fkM^s(\G(\M))$ is concentrated in degrees $-s-1$ and $-s$. We set $I=(a_1, a_2, \dots , a_s)$ for some standard system $a_1, a_2, \dots , a_s$ of parameters of $M$.  Put $f_i=a_it$ for each $1\le i\le s$ and $G=\G(\M)$. 
Let $n$ be an integer with $0\le n\le s$ and put 
$$
G_n=G/(f_1, f_2,\dots f_n)G
.$$ Since $[\H_\fkM^s(G)]_i=(0)$ for all integers $i>-s$ (see, e.g., \cite[Theorem 5.4]{T}), 
it suffices to show $[\Soc \H_\fkM^{s}(G_n)]_i=(0)$ if $i<-s-1$. 
To do this, we will prove $[\Soc \H_\fkM^{s-n}(G_n)]_i=(0)$ if $i<-s+n-1$ by descending induction on $n$. 

When $n=s$, it is clearly true, as the graded module $G_s$ is concentrated in degree $0$. 

Let $n<s$. Put $f=f_{n+1}$ and 
$
\overline{G_n}=G_n/(0):_{G_n}f.
$
Since $\dim\ (0):_{G_n}f=0$, we have $\H_\fkM^{i}(G_n)\cong\H_\fkM^{i}(\overline{G_n})$ for all integers $i>0$, and then applying the local cohomology functor $\H_\fkM^j(\ast)$ to the exact sequence 
$$
0\to \overline{G_n}(-1)\to G_n\to G_{n+1}\to 0
$$ 
of graded $\G(I)$-modules induced by multiplication by $f$ on $G_n$, we get the exact sequence 
$$
\H_\fkM^{s-n-1}(G_n)\xrightarrow{\phi_1}\H_\fkM^{s-n-1}(G_{n+1})\xrightarrow{\phi_2} \H_\fkM^{s-n}(G_n)(-1)\xrightarrow{\phi_3}\H_\fkM^{s-n}(G_n)
$$ of graded local cohomology modules (recall $s-n>0$). Let $i<-s+n-1$. We will show 
$$
[\Soc \H_\fkM^{s-n}(G_n)]_i=(0).
$$
Suppose that there is a nonzero element $z\in [\Soc \H_\fkM^{s-n}(G_n)]_i$. Notice that $\phi_3$ is multiplication by $f$. Since $f \cdot z=0$, there is a nonzero element $y\in [\H_\fkM^{s-n-1}(G_{n+1})]_{i+1}$ such that $\phi_2(y)=z$. By the inductive hypothesis, we have $[\Soc\H_\fkM^{s-n-1}(G_{n+1})]_{i+1}=(0)$, so that $y\not\in\Soc\H_\fkM^{s-n-1}(G_{n+1})$. Hence, there is a homogeneous element $g \in\fkM$ such that $\deg g\le 1$ and $g y\neq 0$. Since 
$$
\phi_2(g y)=g \phi_2(y)=g z =0,
$$ there is a nonzero element $x\in [\H_\fkM^{s-n-1}(G_n)]_{\deg g +i+1}$ such that $\phi_1(x)=g y$. We have the inequality $\deg g +i+1< -s+n+1$, since $i<-s+n-1$ and $\deg g\le 1$. However $\H_\fkM^{s-n-1}(G_n)$ is concentrated in degree $-s+n+1$ (see, e.g., \cite[Lemma 5.3 and Theorem 5.4]{T}). This is a contradiction, and therefore $[\Soc \H_\fkM^{s-n}(G_n)]_i=(0)$. This completes the proof.
\end{proof}

\begin{cor}\label{concentrated}
Assume that the ideal $I$ is generated by a standard system of parameters of $M$. Then the canonical $I$-filtration $\k=\{\k_n\}_{k\in\Z}$ of $K$ with respect to $M$ has the form
$$
\k_n=\left\{
\begin{array}{ll}
I^{n-s}\k_s & \mbox{ if $n\ge s$}\cr
K & \mbox{ if $n< s$}.\cr
\end{array}
\right.
$$
\end{cor}
\begin{proof}
We note $[\H_\fkM^s(\G(\M))]_n=(0)$ for all $n> -s$ and $[\H_\fkM^s(\R(\M))]_n=(0)$ for all $n\le 1-s$ (see, e.g., \cite[Theorem 5.4 and Theorem 6.2]{T}). Since $\mathrm{a}(\G(\M))\le -s$, $\k_n=K $ for all $n< s$ by Lemma \ref{ll}. By the graded local duality theorem, we have
$$
\widehat{C}\otimes_C\Ext_\P^{m+1}(R,\K_\P)\cong[\H_\fkM^s(R)]^\vee
$$ 
as a graded $\P$-module, so that $\G(\k)(-1)\cong\K_{\G(\M)}$ as a graded $\R(I)$-module by  Lemma \ref{canonical filtration}. Then from Proposition \ref{concent} we obtain $\G(\k)(-1)$ is generated by homogeneous elements of degrees $s$ and $s+1$ as a graded $\R(I)$-module, so that the equalities $\k_n=I^{n-s}\k_s+\k_{n+1}$ holds for all integers $n\ge s$. Let $n$ be an integer such that $n\ge s$. Then we see the equality $\k_n=I^{n-s}\k_s+\cap_{m\ge n+1}\k_{m}$, and therefore $\k_n=I^{n-s}\k_s$ (note that $\k$ is a stable $I$-filtration, as the module $L$ is a finitely generated $\R(I)$-module).
\end{proof}

To close this section, let us state the next lemma due to \cite{G83}. For the proof of it, use the same arguments as in the proof of \cite[Lemma (2.15)]{G83}. 

\begin{lem}[\cite{G83}]\label{top}
Assume that $I$ is generated by a standard system $a_1, a_2, \dots , a_s$ of parameters for $M$. Then $$[\H_{\G(I)_+}^s(\G(\M))]_{-s}\cong M/\widetilde{\Sigma}(a_1,a_2,\dots ,a_s; M)$$ as an $A$-module. Hence, $[\K_{\G(\M)}]_s\cong [M/\widetilde{\Sigma}(a_1,a_2,\dots ,a_s; M)]^\vee$ as an $A$-module.
\end{lem}

\section{Preliminary results for the Gorenstein Rees algebra $\R(\q^d)$}
In this section we will summarize some preliminary results for the Gorensteinness of the Rees algebra $\R(\q^d)$. In what follows, we assume that the ring $A$ is a homomorphic image of a Gorenstein local ring. Suppose that $d\ge 2$ and $A$ is an unmixed local ring (i.e., every associated prime ideal of $\widehat{A}$ has same codimension). Hence, $\H^1_\m(A)$ is a finitely generated $A$-module and the canonical module $\K_A$ of $A$ is a faithful $A$-module (see, e.g., \cite[Lemma 3.1]{GNa} and \cite[(1.8)]{A}). We will use the notation $\widetilde{A}$ of Section 2. Recall that $\widetilde{A}$ is an overring of $A$ which is a finitely generated $A$-module such that $\depth_A\widetilde{A}\ge 2$ and $\H^1_\m(A)\cong\widetilde{A}/A$ as an $A$-module.

Let $\fkq$ be a parameter ideal of $A$ generated by elements $a_1,a_2,\dots ,a_d\in A$. We put 
\begin{center}
$R=\R(\q)$ and $S=\R(\q\widetilde{A})$. 
\end{center}
Then the ring $S$ is module-finite over the ring $R$, as $\widetilde{A}$ is module-finite over $A$. 
Let $\psi : C\to A$ be an epimorphism of rings, where $C$ is a Gorenstein local ring of dimension $d$. Put $$\P=C[X_1,X_2,\dots X_d]$$ that is the $\Z$-graded polynomial ring over $C$ such that $\deg X_i=1$ for all $1\le i\le d$ and $\deg c=0$ for all $c\in C$. Let $\varphi :\P\to R$  be the homomorphism of graded $C$-algebras such that $\varphi (X_i)=a_it$ for all $1\le i\le d$ and $\varphi (c)=\psi (c)$ for all $c\in C$. The rings $A$, $\widetilde{A}$, $R$, and $S$ are finitely generated graded $\P$-modules. We set $\K_\P=\P(-d)$,
\begin{align*}
&\K_{\widetilde{A}}=\Ext_{\P}^{d}(\widetilde{A},\K_\P), \\
&\K_R=\Ext_\P^{d-1}(R,\K_\P),\ and \\
&\K_{S}=\Ext_\P^{d-1}(S,\K_\P).
\end{align*}
They are finitely generated graded $\P$-modules. The graded module $\K_{\widetilde{A}}$ is concentrated in degree $0$. Let $X$ be an $\widetilde{A}$-module such that $$X\cong\K_{\widetilde{A}}$$ as an $\widetilde{A}$-module. Notice that $\K_{\widetilde{A}}\cong\K_A$ as an $A$-module because $d\ge 2$ and note that $\mathrm{a}(\G(\q))=-d$ by Lemma \ref{aG}. Let 
$$
\k=\{\k_n\}_{n\in\Z}
$$ 
be a canonical $\q$-filtration of $X$ with respect to $A$ (see Proposition \ref{full canonical filtration}). Then we have $\k_n=X$ for all $n\le d-1$ by Lemma \ref{ll}, in particular, $\k_{d-1}=X$.

Let $\mu_A(\ \ )$ be the number of minimal generators for an $A$-module, $[\quad]^\vee$ the Matlis dual functor with respect to $C$, and $[\quad]^{(i)}$ the $i^{\th}$ Veronese functor for each integer $i$. Let us begin with the following.

\begin{lem}\label{necessary1}
If the Rees algebra $\R(\q^d)$ is a Gorenstein ring, then the following two assertions hold.
\begin{enumerate}
	\item[$(1)$] $\k_d\cong A$ as an $A$-module and hence $\mu_A(\k_d)=1$.
	\item[$(2)$] $X/ \k_d\cong [\K_{\G(\q)}]_d$ as an $A/\q$-module and then $\k_d\subsetneq X$.
\end{enumerate}
\end{lem}

\begin{proof}
Notice that $\R(\fkq^d)=R^{(d)}$ and that $[\H_\fkM^j(R)]^{(i)}\cong\H_{\fkM^{(i)}}^j(R^{(i)})$ as a graded $R^{(i)}$-module for all integers $i,j\ge 0$ (see, e.g., \cite[Proposition (47.5)]{HIO}). 

Since $\K_{R^{(d)}}\cong[\K_R]^{(d)}$ as a graded $R^{(d)}$-module and $[\K_R]_d\cong \k_d$ as an $A$-module, we get $[\K_{\R(\q^d)}]_1\cong \k_d$ as an $A$-module. Since the Rees algebra $\R(\q^d)$ is a Gorenstein, we have 
$$
\K_{\R(\q^d)}\cong\R(\q^d)(-1)
$$ 
as a graded $\R(\q^d)$-module. Looking at the first homogeneous component of it, 
we see $\k_d\cong A$ as an $A$-module.

Since $\R(\fkq^d)\ (=R^{(d)})$ is a Cohen-Macaulay ring and $\dim\R(\fkq^d)=d+1$, we have $[\H_\fkM^d(R)]_{-d}=(0)$, and hence $[\G(\k)]_{d-1}\cong [\K_{\G(\q)}]_d$ as an $A$-module by Lemma \ref {canonical filtration}. Therefore, $X/ \k_d\cong [\K_{\G(\q)}]_d$ as an $A/\q$-module (recall that $X=\k_{d-1}$). From $\mathrm{a}(\G(\q))=-d$ we obtain $[\K_{\G(\q)}]_d\neq (0)$, and then $\k_d\subsetneq X$.
\end{proof}

\begin{lem}\label{kd1}
Assume that $\mu_A(\k_d)=1$. 
Then the following four assertions hold.
\begin{enumerate}
	\item[$(1)$] $\k_d\cong A$ as an $A$-module.
	\item[$(2)$] $X\cong \widetilde{A}$ as an $\widetilde{A}$-module.
	\item[$(3)$] $X/\k_d\cong \widetilde{A}/A$ as an $A$-module.
	\item[$(4)$] $\q\widetilde{A}\subseteq A$.
\end{enumerate}
\end{lem}

\begin{proof}
We write $\k_d=Ay$ for some $y\in \k_d$. Since $X=\k_{d-1}$, we have $\q X\subseteq Ay$. 
\begin{claim}\label{y}
$(0):_{\widetilde{A}}y=(0)$ and then $X=\widetilde{A}y$.
\end{claim}
\begin{proof}
Let $\alpha$ be an element of $\widetilde{A}$ such that $\alpha y=0$. Take an $A$-regular element $a\in\q$ such that $a\H^1_\m(A)=(0)$ (recall that $\H^1_\m(A)$ is a finitely generated $A$-module). Then $a\alpha \in A$, as $a\in (0):_A\H^1_\m(A)=A:\widetilde{A}$. Since $\q X\subseteq Ay$, we get $a\alpha X=(0)$, and hence $a\alpha=0$ because $X$ is a faithful $A$-module (recall that $X\cong\K_{A}$ as an $A$-module and that the ring $A$ is unmixed). Then $\alpha=0$, as $a$ is also an $\widetilde{A}$-regular element, and therefore $(0):_{\widetilde{A}}y=(0)$. 

Let us show $X=\widetilde{A}y$. The inclusion $X\supseteq\widetilde{A}y$ holds, as $X$ is an $\widetilde{A}$-module. Suppose that $X/\widetilde{A}y\neq (0)$. Since $\q X\subseteq Ay\subseteq\widetilde{A}y\subseteq X$, we have $0<\ell_{A}(X/\widetilde{A}y)<\infty$.
Look at the natural exact sequence 
$$
0\to\widetilde{A}y\to X\to X/\widetilde{A}y\to 0
$$
of $A$-modules, and we see $\depth_{A}\widetilde{A}y=1$. But this is a contradiction because $\depth_{A}\widetilde{A}\ge 2$ and $\widetilde{A}\cong \widetilde{A}y$ as an $\widetilde{A}$-module (recall that $(0):_{\widetilde{A}}y=(0)$). Therefore, $X=\widetilde{A}y$. 
\end{proof}

Since $\k_d=Ay$ and $X=\widetilde{A}y$, the assertions (1), (2), and (3) follow from the equality $(0):_{\widetilde{A}}y=(0)$. The assertion (4) implies that $\q\widetilde{A}\subseteq A$, as $\q X\subseteq \k_d$.
\end{proof}

\begin{prop}\label{necessary2}
Assume the Rees algebra $\R(\fkq^d)$ is Gorenstein. Then $\depth A=1$ and $\mathrm{r}_A(A)=1$.
\end{prop}

\begin{proof}
By Lemma  \ref{necessary1} and Lemma \ref{kd1}, we have $$\widetilde{A}/A\cong X/\k_d\neq 0,$$ so that $A\neq\widetilde{A}$ and hence $\depth A=1$ by Proposition \ref{structure}.
To prove $\mathrm{r}_A(A)=1$, it is enough to show $\mathrm{r}_A(\widetilde{A}/A)=1$ (recall that $\widetilde{A}/A\cong\H_\m^1(A)$ as an $A$-module). 
$X/\k_d\cong [\K_{\G(\q)}]_d$ as an $A$-module by Lemma \ref {necessary1}. Since $[\K_{\G(\q)}]_d\cong [[\H_\fkM^d(\G(\q))]_{-d}]^\vee$, we get $$[\H_\fkM^d(\G(\q))]_{-d}\cong [X/\k_d]^\vee$$ as an $A$-module. Thanks to Lemma \ref{top}, we see $[\H_\fkM^d(\G(\q))]_{-d}$ is a cyclic $A$-module and hence so is $[X/\k_d]^\vee$. Therefore, $\mathrm{r}_A(\widetilde{A}/A)=1$, as $\widetilde{A}/A\cong X/\k_d$. 
\end{proof}

Since $\H_\m^1(A)\cong \widetilde{A}/A$ as an $A$-module, the next result directly follows from Lemma \ref{kd1} and \cite[Corollary 3.7]{T}.

\begin{lem}\label{standard}
Assume that the Rees algebra $\R(\fkq^d)$ is Gorenstein. Then the system $a_1,a_2,\dots ,a_d$ of parameters for $A$ is standard if $\H^i_\m(A)=(0)$ for $i\neq 1,d$.
\end{lem}

In the rest of this section, we assume that the sequence $a_1,a_2,\dots ,a_d$ is a standard system of parameters for $A$. Then we note the Rees algebra $\R(\fkq^d)$ is Cohen-Macaulay by \cite[Theorem 6.2]{T}. 

\begin{lem}\label{iff}
The Rees algebra $\R(\fkq^d)$ is Gorenstein if and only if $\mu_A(\k_d)=1$.
\end{lem}

\begin{proof}
The only if part directly follows from Lemma \ref{necessary1}. Conversely, assume that $\mu_A(\k_d)=1$. Let $y$ be an element of $\k_d$ such that $\k_d=Ay$. Then $X=\widetilde{A}y$ by the claim above. Hence, we have 
$$
\k_n=\left\{
\begin{array}{ll}
\q^{n-d} y & \mbox{ if $n\ge d$}\cr
\widetilde{A} y & \mbox{ if $n< d$}\cr
\end{array}
\right.
$$
by Corollary \ref{concentrated}. Let $f:R^{(d)}(-1)\to [\R(\k)_+]^{(d)}$ be the homomorphism of graded $R^{(d)}$-modules such that $f(1)=t^{d}\otimes y$. Then we see $f$ is bijective because of the equalities above (recall that $(0):_{\widetilde{A}}y=(0)$ by Claim \ref{y}).
Since $$[\R(\k)_+]^{(d)}\cong [\K_R]^{(d)}\cong \K_{R^{(d)}}$$ as a graded $R^{(d)}$-module, we obtain that $\R(\q^d)(-1)\cong\K_{\R(\q^d)}$ as a graded $\R(\q^d)$-module (recall that $\R(\q^d)=R^{(d)}$). Therefore, $\R(\fkq^d)$ is a Gorenstein ring.
\end{proof}

We recall that $a_1\widetilde{A}=\cup_{n>0}\ a_1A:_A \m^n$, which is a common ideal of $A$ and $\widetilde{A}$. We put $$A'=A/a_1\widetilde{A}.$$ The sequence $a_2,a_3,\dots ,a_d$ forms a standard system of parameters for $A'$ (see \cite[Proposition 2.2]{H82}).

\begin{lem}\label{ind}
$[\K_{\R(\q A')}]_{d-1}\cong\k_{d}/a_1X$ as an $A'$-module.
\end{lem}

\begin{proof}
We consider the epimorphism $\varepsilon : R\to\R(\q A')$ of Rees algebras induced by the natural epimorphism $A\to A'$ of rings. Put $\calK=\Ker \varepsilon $. Since the sequence $a_1,a_2,\dots ,a_d$ is a standard system of parameters for $A$, we obtain $$a_1\widetilde{A}\cap\q^{n+1}=a_1\q^{n}$$ for all integers $n\ge 0$ from \cite[Theorem (2.4)]{G83}, and therefore $$\calK=(a_1\widetilde{A})R+a_1tR.$$ Define a graded $R$-homomorphism $f:R(-1)\to\calK$ carrying $1$ to $a_1t$. Put $\calC=\Coker f$. Then $\calC =[\calC ]_0\cong \widetilde{A}$ as an $R$-module. Hence, we have two exact sequences\begin{center}
 $0\to\calK\xrightarrow{i} R\to\R(\q A')\to 0$ and $0\to R(-1)\xrightarrow{f} \calK\to\widetilde{A}\to 0$ 
\end{center}
of graded $R$-modules. Taking $\K_\P$-dual of them, we get the exact sequences
$$
0\to\K_R\to\Ext_\P^{d-1}(\calK,\K_\P)\to\K_{\R(\q A')}\to\Ext_\P^{d}(R,\K_\P)
$$
and
$$
0\to\Ext_\P^{d-1}(\calK,\K_\P)\to\K_R(1)\to\Ext_\P^{d}(\widetilde{A},\K_\P)
$$
of graded $R$-modules. Since $\mathrm{a}(R)=\mathrm{a}(\R(\q A'))=-1$, we obtain $[\Ext_\P^{d-1}(\calK,\K_\P)]_n=(0)$ for all integers $n\le 0$. Since the graded $R$-module $\Ext_\P^{d}(\widetilde{A},\K_\P)$ is concentrated in degree $0$, we get $$\Ext_\P^{d-1}(\calK,\K_\P)\cong[\K_R]_{\ge2}(1)$$ as a graded $R$-module. Thus we get the exact sequence
$$
0\to\K_R\to [\K_R]_{\ge2}(1)\to\K_{\R(\q A')}\to\Ext_\P^{d}(R,\K_\P)
$$
of graded $R$-modules. The above map $\K_R\to [\K_R]_{\ge2}(1)$ is induced by multiplication by $a_1t$ on $\K_R$, as so is $i\circ f$. Looking at the $(d-1)^{\th}$ degree of homogeneous component of the exact sequence above, we get the exact sequence
$$
0\to X\to \k_{d}\to [\K_{\R(\q A')}]_{d-1}\to 0
$$ 
of $A$-modules (recall that $\widehat{C}\otimes_C\Ext_\P^{d}(R,\K_\P)\cong[\H_\fkM^d(R)]^\vee$ as a graded $R$-module and that $[\H_\fkM^d(R)]_{1-d}=(0)$ by \cite[Theorem 6.2]{T}). Therefore, we get $\k_{d}/a_1X\cong[\K_{\R(\q A')}]_{d-1}$ as an $A'$-module.
\end{proof}

\begin{lem}\label{ind3}
Assume that $d\ge 3$ and $\K_{\widetilde{A}}\cong \widetilde{A}$ as an $\widetilde{A}$-module. Then $\R(\fkq^d)$ is a Gorenstein ring if and only if so is $\R(\fkq^{d-1}A')$.
\end{lem}

\begin{proof}
We may assume $X=\widetilde{A}$, since $X\cong\K_{\widetilde{A}}$ as an $\widetilde{A}$-module. Recall $\k_{d-1}=X$, and we get the inclusion $\q\widetilde{A}\subseteq \k_d$. Hence, $\height_{\widetilde{A}}(\k_d\widetilde{A})=d\ge 3$ whenever $\k_d\widetilde{A}\neq \widetilde{A}$. 

Suppose that the ring $\R(\fkq^{d-1}A')$ is Gorenstein. Then $$[\K_{\R(\q A')}]_{d-1}\cong A'$$ as an $A'$-module by Lemma \ref{necessary1}. We have $[\K_{\R(\q A')}]_{d-1}\cong\k_{d}/a_1\widetilde{A}$ as an $A$-module by Lemma \ref{ind}, so that $$A'\cong \k_{d}/a_1\widetilde{A}$$ as an $A$-module, and hence we can write $\k_d=Ay+a_1\widetilde{A}$ for some $y\in\k_d$. Then $$\k_d\widetilde{A}=(y,a_1)\widetilde{A},$$ and therefore the element $y$ must be a unit of $\widetilde{A}$. Then $\k_d=Ay+a_1\widetilde{A}y=Ay$, as $a_1\widetilde{A}\subseteq A$. Therefore, $\R(\fkq^d)$ is a Gorenstein ring by Lemma \ref{iff}.

Conversely, suppose that the ring $\R(\fkq^d)$ is Gorenstein. Then $\mu_A(\k_d)=1$ by Lemma \ref{necessary1}, so that $\mu_{A'}([\K_{\R(\q A')}]_{d-1})=1$ by Lemma \ref{ind}. Therefore, $\R(\fkq^{d-1}A')$ is a Gorenstein ring by Lemma \ref{iff}.
\end{proof}

Notice that $\mathrm{a}(\G(\q \widetilde{A}))=-d<0$ by Lemma \ref{aG}.
Let $\w=\{\w_n\}_{n\in\Z}$ be a canonical $\q$-filtration of $X$ with respect to $\widetilde{A}$ such that $\w_n\subseteq\k_n$ for all $n\in\Z$ (see Lemma \ref{two canonical filtrations}) and the diagram
$$
\begin{CD}
\K_{S}@>>>\K_{R}\\
@VV\wr V     @VV\wr V \\
\R(\w)_+@>>>\R(\k)_+\\
\end{CD}
$$
of graded $R$-modules is commutative for some vertical isomorphisms, where the top homomorphism is induced by the inclusion map $R\hookrightarrow S$ and the bottom one is an inclusion map. Note that $X=\k_n=\w_n$ for all $n\le d-1$ by Lemma \ref{ll}. The next result is the key lemma to investigate the Gorensteinness of $\R(\q^d)$ in this paper.

\begin{lem}\label{kw}
$\R(\k)/\R(\w)\cong \P\otimes_C[\widetilde{A}/A]^\vee(-d)$ as a graded $\P$-module. In particular, one has $\k_d/\w_d\cong [\widetilde{A}/A]^\vee$ as an $A$-module.
\end{lem}

\begin{proof}
We have  the exact sequence $0\to R\to S\to \P\otimes_C\widetilde{A}/A\to 0$ of graded $\P$-modules by Corollary \ref{independent2}. Taking $\K_\P$-dual of it, we get the exact sequence
$$
0\to \K_{S}\to \K_R\to \Ext_\P^d(\P\otimes_C\widetilde{A}/A,\K_\P)\to\Ext_\P^d(R,\K_\P)
$$
of graded $\P$-modules. The local duality theorem implies that $$\Ext_\P^d(\P\otimes_C\widetilde{A}/A,\K_\P)\cong \P\otimes_C[\widetilde{A}/A]^\vee(-d)$$ as a graded $\P$-module. Besides, we obtain the equalities $$[\Ext_\P^d(R,\K_\P)]_n=(0)$$ for all $n\ge d-1$ from \cite[Theorem 6.2]{T}.
Therefore, we get the exact sequence $$0\to \K_{S}\to \K_R\to \P\otimes_C[\widetilde{A}/A]^\vee(-d)\to 0$$ of graded $\P$-modules. Then the exact sequence 
$$
0\to \R(\w)_+\xrightarrow{i} \R(\k)_+\to \P\otimes_C[\widetilde{A}/A]^\vee(-d)\to 0
$$
of graded $\P$-modules follows from the diagram above, where $i$ stands for an inclusion map, and thus we obtain the required isomorphism (recall that $\w_n=\k_n$ for all integers $n\le d-1$).
\end{proof}

We put $\fka =\widetilde{\Sigma}(a_1,a_2,\dots ,a_d; A)$ that is a common ideal of $A$ and $\widetilde{A}$ (see Corollary \ref{independent}). Then a canonical $\q$-filtration of $\widetilde{A}$ with respect to $\widetilde{A}$ is uniquely determined as follows.

\begin{lem}\label{w}
Let $X=\widetilde{A}$. Then a canonical $\q$-filtration $\w=\{\w_n\}_{n\in\Z}$ of $\widetilde{A}$ with respect to $\widetilde{A}$ has the form
$$
\w_n=\left\{
\begin{array}{ll}
\q^{n-d}\fka & \mbox{ if $n\ge d$}\cr
\widetilde{A} & \mbox{ if $n< d$}.\cr
\end{array}
\right.
$$
\end{lem}

\begin{proof}
The sequence $a_1,a_2,\dots ,a_d$ is a standard system of parameters for $\widetilde{A}$ by Corollary \ref{independent}, and hence $[\H_\fkM^d(\R(\q\widetilde{A}))]_d=(0)$ by \cite[Theorem 6.2]{T}. Then we obtain from Lemma \ref{canonical filtration} that $\G(\w)_{d-1}\cong[\K_{\G(\q\widetilde{A})}]_d$ as an $A$-module, and hence $\widetilde{A}/\w_{d}\cong[\K_{\G(\q\widetilde{A})}]_d$ as an $A$-module (recall that $\w_{d-1}=\widetilde{A}$). Thanks to Lemma \ref{top}, we have $[\H_\fkM^d(\G(\q\widetilde{A}))]_{-d}\cong \widetilde{A}/\fka$ as an $A$-module, and therefore $\widetilde{A}/\w_d\cong[\widetilde{A}/\fka]^\vee$ as an $A$-module because $\K_{\G(\q\widetilde{A})}\cong [\H_\fkM^d(\G(\q\widetilde{A}))]^\vee$ as a graded $R$-module. Then since $\fka\cdot\widetilde{A}/\w_d=(0)$, we have $\w_d\supseteq\fka$, and furthermore we see $\w_d=\fka$, as $\ell_A(\widetilde{A}/\w_d)=\ell_A(\widetilde{A}/\fka)<\infty$. Thus we obtain the required equalities from Corollary  \ref{concentrated}.
\end{proof}

We set $\fkc=(0):_{A}\H^1_\m(A)$. Recall that $\fkc=A:\widetilde{A}$ and that the inclusion $\fka\subseteq \fkc$ holds. If the ring $\R(\q^d)$ is Gorenstein, then the equality $\fka=\fkc$ follows from the next result.

\begin{prop}\label{necessary3}
If $\mu_A(\k_d)=1$, then $[\widetilde{A}/A]^\vee\cong A/\fka$ as an $A$-module and hence $\fka=\fkc$.
\end{prop}

\begin{proof}
We may assume that $X=\widetilde{A}$ by Lemma \ref{kd1}. Then $\w_d=\fka$ by Lemma \ref{w}. We write $\k_d=Ay$ for some $y\in\k_d$. The inclusion $\fka\subseteq Ay$ follows from $\w_d\subseteq\k_d$. Then $\fka\subseteq \widetilde{A}y$, so that we see $y$ is a unit of the ring $\widetilde{A}$, as $d\ge 2$. Therefore, $\fka=\fka y$ and hence $\k_d/\w_d=Ay/\fka y$. Since $Ay/\fka y\cong A/\fka$ as an $A$-module, we obtain from Lemma \ref{kw} that $[\widetilde{A}/A]^\vee\cong A/\fka$ as an $A$-module.
Since $(0):_A[\widetilde{A}/A]^\vee=(0):_A \widetilde{A}/A$ and $\widetilde{A}/A\cong\H^1_\m(A)$, we have $\fka=\fkc$.
\end{proof}

\begin{thm}\label{n=d}
The ring $\R(\q^n)$ is not Gorenstein for any positive integer $n\neq d$ whenever $\depth A=1$ and $\q$ is a standard parameter ideal of $A$.
\end{thm}
\begin{proof}
Suppose that the ring $\R(\q^n)$ is Gorenstein for some integer $n>0$. Since $\K_{\R(\q^n)}\cong\R(\k)^{(n)}$ as a graded $\R(\q^n)$-module, we have $$\R(\q^n)(-1)\cong\R(\k)^{(n)}$$ as a graded $\R(\q^n)$-module. Hence, $\k_n\cong A$ as an $A$-module.  

Let $n<d$. Then $\k_n=X$ by Lemma \ref{ll}, and hence $A\cong\K_A$ as an $A$-module (recall that $X\cong \K_A$ as an $A$-module). This is contradiction to the equality $\depth A=1$. 

Let $n>d$. From Lemma \ref{kw} we obtain that $$\R(\k)/\R(\w)\cong \P\otimes_C[\widetilde{A}/A]^\vee(-d)$$ as a graded $\P$-module. Looking at the $n^{\th}$ degree of homogeneous component of it, we get a surjection $A\to ([\widetilde{A}/A]^\vee)^\mu $ of $A$-modules, where $\mu =\binom{n-1}{d-1}$. Since $\depth A=1$, we have $\widetilde{A}/A\neq (0)$ by Proposition \ref{structure}. This is contradiction to the inequality $\mu >1$ (recall that we assume $d\ge 2$).
\end{proof}

\section{Proof of the main theorem}

We are now ready to prove Theorem \ref{main} and Corollary \ref{Bbm}. Let $\fkq$ be a parameter ideal of $A$ generated by elements $a_1,a_2,\dots ,a_d\in A$. We put $\fka =\widetilde{\Sigma}(a_1,a_2,\dots ,a_d; A)$ and $\mathfrak{c}=(0):_A \H^1_\m(A)$.

\begin{proof}[Proof of Theorem \ref{main}]
We may assume the ring $A$ is complete. Since $\H^i_\m(A)=(0)$ for $i\neq 1, d$, and $\ell_A (\H^1_\m(A))<\infty$, we may assume the sequence $a_1,a_2,\dots , a_d$ forms a standard system of parameters for $A$, as $\q\H_\m^1(A)=(0)$ (see Lemma \ref{standard} and \cite[Corollary 3.7]{T}). We have $A$ is unmixed and $\widetilde{A}$ is Cohen-Macaulay (see, e.g., \cite[Proposition 3]{SV} and Corollary \ref{1980}). Then the sequence $a_1,a_2,\dots , a_d$ forms a regular sequence on $\widetilde{A}$, so that $\fka=\q\widetilde{A}$ by Corollary  \ref{independent}.

The implication (1)$\Rightarrow $(2) follows from Proposition \ref{necessary2} and Proposition  \ref{necessary3}. The equivalence (2)$\Leftrightarrow$(3) is clear by  Proposition \ref{proj}. The last assertion follows from Lemma \ref{r=1} and Theorem \ref{gor}. We shall show the implication (2)$\Rightarrow $(1). We will use induction on $d$. Put $$A'=A/a_1\widetilde{A}.$$ Recall that the sequence $a_2,a_3,\dots ,a_d$ forms a standard system of parameters for $A'$.

Let $d=2$. Then $A'$ is an $1$-dimensional Gorenstein local ring by Proposition \ref{proj}, and hence $$[\K_{\R(\q A')}]_{1}\cong A'$$ because $\R(\q A')$ is the polynomial ring over $A'$ in one variable. Notice that $\K_A\cong\widetilde{A}$ as an $A$-module (see Theorem \ref{gor}). Let $\k=\{\k_n\}_{n\in\Z}$ be a canonical $\q$-filtration of $\widetilde{A}$ with respect to $A$ (see Proposition \ref{full canonical filtration}). We have $$[\K_{\R(\q A')}]_{1}\cong \k_{2}/a_1\widetilde{A}$$ as an $A$-module by Lemma \ref{ind}, and hence we get an isomorphism
$$\eta:A'\xrightarrow{\sim} \k_{2}/a_1\widetilde{A}$$ of $A$-modules. Write $\eta(1)=\overline{y}$ for some $y\in\k_2$. Then $$\k_2=Ay+a_1\widetilde{A}.$$ We want to show $\k_2=Ay$.
The inclusion $\eta(\fka/a_1\widetilde{A})\subseteq\fka/a_1\widetilde{A}$ holds because $\fka y\subseteq \fka$ (recall that $y\in\widetilde{A})$. We consider the commutative and exact diagram
$$
\begin{CD}
0@>>> \fka/a_1\widetilde{A}@>>> \k_2/a_1\widetilde{A}@>>> [\widetilde{A}/A]^\vee@>>> 0\\
@.@AA\eta_1 A@A\wr A\eta A@AA\eta_2A\\
0@>>> \fka/a_1\widetilde{A}@>>> A/a_1\widetilde{A}@>>> A/\fka@>>> 0
\end{CD}
$$
of $A$-modules, where the top exact sequence is induced by
the isomorphism $$\k_d/\w_d\xrightarrow{\sim} [\widetilde{A}/A]^\vee$$ in Lemma \ref{kw} (recall that $\w_d=\fka$ by Lemma \ref{w}), the bottom is the natural one, and $\eta_1$ is induced by the inclusion $\eta(\fka/a_1\widetilde{A})\subseteq\fka/a_1\widetilde{A}$. The equality $$\ell_A (A/\fka) = \ell_A (\widetilde{A}/A)$$ follows from Lemma \ref{r=1}, as $\fkc =\fka$. Then the surjection $\eta_2$ must be bijective and hence so is $\eta_1$. Then $\eta(\fka/a_1\widetilde{A})=\fka/a_1\widetilde{A}$ and therefore $$\fka=\fka y+a_1\widetilde{A}.$$ 

On the other hand, we have $\fka=(a_1, a_2)\widetilde{A}$. Hence, $$(a_1, a_2)\widetilde{A}=(a_1, a_2y)\widetilde{A},$$ so that we can write $a_2=a_1\alpha_1+a_2y\alpha_2$ for some $\alpha_1, \alpha_2\in \widetilde{A}$. Since $a_1,a_2$ is a regular sequence on $\widetilde{A}$, we have $1-y\alpha_2\in a_1\widetilde{A}$. This means the element $y$ must be a unit of $\widetilde{A}$, as $a_1\widetilde{A}$ is contained in the Jacobson radical of $\widetilde{A}$.
Then $$\k_2=Ay+a_1\widetilde{A}=Ay+a_1\widetilde{A}y=Ay,$$ as $a_1\widetilde{A}\subseteq A$. Therefore, $\k_2=Ay$, and thus $\R(\fkq^2)$ is a Gorenstein ring by Lemma \ref{iff}.

Let $d\ge 3$. It follows that $\H^i_\m(A')=(0)$ for $i\neq 1,d$ and $0<\ell_{A'}(\H^1_\m(A'))<\infty$ from applying the local cohomology functor $\H_\m^i(\ast )$ to the natural exact sequence
$$
0\to A'\to\widetilde{A}/a_1\widetilde{A}\to\widetilde{A}/A\to 0
$$
of $A$-modules (recall that $\widetilde{A}$ is a Cohen-Macaulay $A$-module and $\widetilde{A}/A\cong\H_\m^1(A)$ as an $A$-module). Then we have $$\H^1_\m(A)\cong\H^1_\m(A')$$ as an $A$-module, and therefore $\mathrm{r}_{A'}(\H^1_\m(A'))=\mathrm{r}_A(\H^1_\m(A))=1$. Put $\fkc'=(0):_{A'}\H_\m^1(A')$. Then $\fkc'=\fkc A'=\fka A'$, as $\fkc=\fka$. Corollary \ref{independent} implies
$$
\fka A'=\widetilde{\Sigma}(a_1,a_2,\dots ,a_d; \widetilde{A})/a_1\widetilde{A}=\widetilde{\Sigma}(a_2,a_3,\dots ,a_d; \widetilde{A}/a_1\widetilde{A})
.$$ 
The equality $\widetilde{A}/a_1\widetilde{A}=\widetilde{A'~}$ follows from Corollary \ref{minimality}, so that $$\widetilde{\Sigma}(a_2,a_3,\dots ,a_d; \widetilde{A}/a_1\widetilde{A})=\widetilde{\Sigma}(a_2,a_3,\dots ,a_d; A')$$ by Corollary \ref{independent}. Thus the equality $$\fkc'=\widetilde{\Sigma}(a_2,a_3,\dots ,a_d; A')$$ holds. By the inductive hypothesis, the ring $\R(\fkq^{d-1}A')$ is Gorenstein, and therefore so is the ring $\R(\fkq^{d})$ by Lemma \ref{ind3}.
\end{proof}

\

\begin{proof}[Proof of Corollary \ref{Bbm}]
We may assume that $A$ is a complete Buchsbaum local ring of depth one. 
Put 
$
\fkb=(a_1,a_2,\dots,a_{d-1})A:_Aa_d+a_dA.
$
Then $\e_\q(A)=\ell_A (A/\fkb)$ by \cite[Proposition (3.7)]{G83}.
Hence, the natural exact sequence
$$0\to \fka/\fkb\to A/\fkb\to A/\fka\to 0$$
of $A$-modules implies the equality $\ell_A (\fka/\fkb)=\e_\q(A)-\ell_A (A/\fka)$, and then we obtain 
$$
\ell_A (\fka/\fkb)
=
\displaystyle\sum_{i=1}^{d-1}
\begin{pmatrix}
d-1\\ i-1
\end{pmatrix}
\ell_A (\H^i_\m(A))
$$
from \cite[Proposition (3.6)]{G83A} (recall that
$
\ell_A (A/\q)-\e_\q (A)=\sum_{i=1}^{d-1}
\binom{d-1}{i}\ell_A (\H^i_\m(A))$). We note $\fka/\fkb\neq (0)$, as $\depth A=1$ and $d\ge 2$.

Suppose $\e_\m (A)=2$ and $\q$ is a reduction of $\m$. Since $$\e_\m(A)=\e_\q(A)=\ell_A (A/\fka)+\ell_A (\fka/\fkb),$$ we have $\ell_A (A/\fka)=1$ and $\ell_A (\fka/\fkb)=1$, so that $\fka=\m$ and $\ell_A (\H^1_\m(A))=1$. Hence, $\fka=\fkc$ and $\mathrm{r}_A(\H^1_\m(A))=1$. Therefore, $\R(\fkq^d)$ is a Gorenstein ring by Theorem \ref{main}.

Conversely, suppose $\R(\fkq^d)$ is a Gorenstein ring. Then $A/\fkb$ is a Gorenstein local ring by \cite[Lemma (3.6)]{G83}. Since the ring $A$ is Buchsbaum, the inclusion $\fkm\fka\subseteq \fkb$ holds, so that $(0):_{\fka/\fkb}\m =\fka/\fkb$. Hence, $\ell_A (\fka/\fkb)=1$ because $\mathrm{r}_A(A/\fkb)=1$, so that $\ell_A (\H^1_\m(A))=1$ and $\H^i_\m(A)=(0)$ if $i\neq 1,d$. Then $\m=\fkc$, and hence $\m=\fka$ (note that $\fka=\fkc$ by Theorem \ref{main}). Hence, we see $\q$ is a reduction of $\m$, as $\fka=\q\widetilde{A}$. Since $\e_\q(A)=\ell_A (A/\fka)+\ell_A (\fka/\fkb)$, we get $\e_\m(A)=2$, as $\e_\m(A)=\e_\q(A)$. 
\end{proof}

Concluding this paper, let us give the following example to illustrate Theorem \ref{main}.

\begin{ex}\label{ex}
Let $(B,\n)$ be a $d$-dimensional Gorenstein local ring and $Q$ a parameter ideal of $B$. 
Set $A=B\ltimes Q$ that is the idealization of $Q$ over $B$ and put $\q=QA$ that is a parameter ideal of $A$. Then $\R(\q^d)$ is a Gorenstein ring whenever $d\ge 2$.
\end{ex}

\begin{proof}
Let $d\ge 2$ and put $\m=\n\times Q$. We set $X=B\ltimes B$  that is the idealization of $B$ over $B$. Then $X$ is an $A$-algebra that is finite as an $A$-module. Look at the natural exact sequence $$0\to A \to X\to X/A\to 0$$ of $A$-modules, and we get $\H^i_\m(A)=(0)$ for $i\neq 1,d$, and $\ell_A(\H^1_\m(A))<\infty$ because $\depth_AX=d$. Then $X=\widetilde{A}$ by Proposition \ref{structure}. Notice that $\widetilde{A}/A=X/A\cong B/Q$ as an $A$-module. Since $\H^1_\m(A)\cong \widetilde{A}/A$ as an $A$-module, we see $\depth A=1$ and $\mathrm{r}_A(\H^1_\m(A))=\mathrm{r}_{A}(B/Q)=\mathrm{r}_{B}(B/Q)=1$ because $B$ is a Gorenstein local ring and $B/\n\cong A/\m$ as a ring. Therefore, $\mathrm{r}_A(A)=1$. Put $\fkc=(0):_A\H^1_\m(A)$. Then $$\fkc=Q\times Q=Q\widetilde{A},$$ so that we see $\q$ is a reduction of $\fkc$. We have $$\e_\fkc (A)=\e_\q (A)=\e_\q (\widetilde{A})=\ell_A(\widetilde{A}/Q\widetilde{A})=\ell_A(\widetilde{A}/\fkc).$$ The equality $\ell_A (\widetilde{A}/\fkc)=2\ell_A (A/\fkc)$ follows from Lemma \ref{r=1}, so that $\e_\fkc (A)=2\ell_A(A/\fkc)$. Thus $\R(\q^d)$ is a Gorenstein ring by Theorem \ref{main}.
\end{proof}

\end{document}